\documentclass[article]{siamonline1116}
\usepackage{dsfont}

\usepackage{lipsum}
\usepackage{amsfonts}
\usepackage{graphicx}
\usepackage{epstopdf}
\usepackage{algorithmic}
\ifpdf
  \DeclareGraphicsExtensions{.eps,.pdf,.png,.jpg}
\else
  \DeclareGraphicsExtensions{.eps}
\fi

\numberwithin{theorem}{section}

\newcommand{\TheTitle}{Rigorous Analysis for Efficient Statistically Accurate Algorithms for Solving Fokker-Planck Equations in Large Dimensions}
\newcommand{\TheAuthors}{Nan Chen, Andrew J. Majda, and Xin. T. Tong}

\headers{\TheTitle}{\TheAuthors}

\title{{\TheTitle}\thanks{Submitted to the editors DATE.
\funding{The research of A.J.M. is partially supported by the Office of Naval Research Grant ONR MURI N00014-16-1-2161 and the Center for Prototype Climate Modeling (CPCM) at New York University Abu Dhabi Research Institute. N.C. is supported as a postdoctoral fellow through A.J.M's ONR MURI Grant. X.T.T is supported by NUS grant R-146-000-226-133.}}}

\author{
  Nan Chen\thanks{Department of Mathematics and Center for Atmosphere
Ocean Science, Courant Institute of Mathematical Sciences, New York University, New York, NY, USA
    (\email{chennan@cims.nyu.edu}).}
  \and
  Andrew J. Majda\thanks{Department of Mathematics and Center for Atmosphere
Ocean Science, Courant Institute of Mathematical Sciences, New York University, New York, NY, USA and Center for Prototype
Climate Modeling, New York University Abu Dhabi, Saadiyat Island, Abu Dhabi, UAE. (\email{jonjon@cims.nyu.edu}).}
  \and
  Xin T. Tong\thanks{Department of Mathematics, National University of Singapore, Singapore
 (\email{mattxin@nus.edu.sg}).}
}

\usepackage{amsopn}

\ifpdf
\hypersetup{
  pdftitle={\TheTitle},
  pdfauthor={\TheAuthors}
}
\fi

\newcommand{\calO}{\mathcal{O}}

\newcommand{\E}{\mathbb{E}}
\newcommand{\calC}{\mathcal{C}}
\newcommand{\calE}{\mathcal{E}}
\newcommand{\bfu}{\mathbf{u}}
\newcommand{\bfB}{\mathbf{B}}
\newcommand{\bfF}{\mathbf{F}}

\newcommand{\reals}{\mathbb{R}}

\newcommand{\unit}{\mathds{1}}

\newcommand{\bfP}{\mathbf{P}}
\newcommand{\calF}{\mathcal{F}}

\newcommand{\calD}{\mathcal{D}}

\newcommand{\one}{\mathbf{I}}
\newcommand{\two}{\mathbf{II}}

\newcommand{\uone}{\mathbf{u}_\one}
\newcommand{\utwo}{\mathbf{u}_\two}

\newcommand{\barut}{\bar{\mathbf{u}}_\two}
\newcommand{\bfA}{\mathbf{A}}
\newcommand{\bfC}{\mathbf{C}}
\newcommand{\bfa}{\mathbf{a}}

\newcommand{\bfv}{\mathbf{v}}

\newcommand{\bfR}{\mathbf{R}}
\newcommand{\bfW}{\mathbf{W}}

\newcommand{\bfSigma}{\mathbf{\Sigma}}

\newtheorem{thm}{Theorem}[section]
\newtheorem{cor}[thm]{Corollary}
\newtheorem{lem}[thm]{Lemma}
\newtheorem{prop}[thm]{Proposition}

\newtheorem{aspt}[thm]{Assumption}
\newtheorem{rem}[thm]{Remark}

\begin{document}

\maketitle

\begin{abstract}
This article presents a rigorous analysis for efficient statistically accurate algorithms for solving the Fokker-Planck equations associated with high-dimensional nonlinear turbulent dynamical systems with conditional Gaussian structures. Despite the conditional Gaussianity, these nonlinear systems contain many strong non-Gaussian features such as intermittency and fat-tailed probability density functions (PDFs). The algorithms involve a hybrid strategy that requires only a small number of samples $L$ to capture both the transient and the equilibrium  non-Gaussian PDFs with high accuracy.
Here, a conditional Gaussian mixture in a high-dimensional subspace via an extremely efficient parametric method is combined with a judicious Gaussian kernel density estimation in the remaining low-dimensional subspace.
Rigorous analysis shows that the mean integrated squared error in the recovered PDFs in the  high-dimensional subspace is bounded by the inverse square root of the determinant of the conditional covariance, where the conditional covariance is completely determined by the underlying dynamics and is independent of $L$. This is fundamentally different from a direct application of  kernel methods to solve the full PDF, where $L$ needs to increase exponentially with the dimension of the system and the bandwidth shrinks.
A detailed comparison between different methods justifies that the efficient statistically accurate algorithms are able to overcome the curse of dimensionality. It is also shown with mathematical rigour that these algorithms are robust in long time provided that the system is controllable and stochastically stable. Particularly, dynamical systems with energy-conserving quadratic nonlinearity as in many geophysical and engineering turbulence are proved to have these properties.

\end{abstract}

\begin{keywords}
  Fokker-Planck equation, high-dimensional non-Gaussian PDFs, hybrid strategy, small sample size, long time persistence
\end{keywords}

\begin{AMS}
  35Q84, 76F55, 65C05, 37C75, 93B05
\end{AMS}

\section{Introduction}
The Fokker-Planck equation is a partial differential equation (PDE) that governs the time evolution of the probability density function (PDF) of a complex system with noise \cite{gardiner1985stochastic, risken1989fokker}. Many complex dynamical systems in geophysical and engineering turbulence, neuroscience and excitable media have large dimensions and strong nonlinearities, the associated PDFs of which are highly non-Gaussian with intermittency and extreme events \cite{majda2016introduction, lindner2004effects}. Predicting the rare and extreme events \cite{chen2014predicting, cousins2014quantification, ghil2011extreme, palmer2002quantifying, mohamad2016probabilistic, cousins2016reduced, thual2016simple}, quantifying the uncertainty in the presence of intermittent instabilities \cite{majda2012lessons, branicki2012quantifying, greco2009statistical, branicki2014non} and characterizing other non-Gaussian features \cite{neelin2010long, huang2001application} all require solving high-dimensional Fokker-Planck equations with strong non-Gaussian features.

Since there is no general closed-form solution for the Fokker-Planck equation, various numerical and approximate approaches have been developed to solve the evolution of the PDF $p(\mathbf{u},t)$, where $\mathbf{u}$ consists of the state variables and $t$ is the time. However, traditional numerical methods such as finite element and finite difference as well as the direct Monte Carlo simulations of the underlying dynamics all suffer from the curse of dimensionality \cite{robert2004monte, daum2003curse, pichler2013numerical, kumar2006solution, spencer1993numerical}. Furthermore, even in the low-dimensional scenarios, substantial computational cost is already required for an accurate estimation of the fat tails of the highly intermittent non-Gaussian PDFs.
On the other hand, different methods for solving the partial or the approximate solutions of $p(\mathbf{u},t)$ have been proposed for special dynamical systems. For example, asymptotic expansion with truncations provides good approximate PDFs associated with the slow varying variables in non-Gaussian systems with multiscale features \cite{gardiner1985stochastic, majda1999models, majda2001mathematical, majda2006stochastic}.
Splitting methods \cite{er2011methodology, er2012state},  orthogonal functions and  tensor decompositions \cite{von2000calculation, sun2014numerical, risken1989fokker} are able to provide reasonably good estimations of the steady state PDFs. If the systems are weakly nonlinear with additive noise, then equivalent linearization method \cite{socha2007linearization, booton1954nonlinear} is also frequently used for solving approximate solutions.

In recent work by two of the authors \cite{chen2017efficient},  efficient statistically accurate algorithms have been developed for solving the Fokker-Planck equation associated with high-dimensional nonlinear turbulent dynamical systems with conditional Gaussian structures \cite{chen2016filtering}. Decomposing the state variables $\mathbf{u}$ into two groups $\mathbf{u}=(\mathbf{u}_\mathbf{I},\mathbf{u}_\two)$ with $\mathbf{u}_{\mathbf{I}}\in R^{N_{\mathbf{I}}}$ and $\mathbf{u}_{\two}\in R^{N_{\two}}$. The conditional Gaussian systems are characterized by the fact that once a single trajectory of $\mathbf{u}_\mathbf{I}(s\leq t)$ is given, $\mathbf{u}_\two(t)$ conditioned on $\mathbf{u}_\mathbf{I}(s\leq t)$ becomes a Gaussian process. Despite the conditional Gaussian structure, the coupled system of $\mathbf{u}_\mathbf{I}$ and $\mathbf{u}_\two$ is highly nonlinear and  it is able to capture many strong non-Gaussian features such as intermittency and fat-tailed PDFs that are commonly seen in nature \cite{chen2016filtering}. Note that in most turbulent dynamical systems, the observed variables $\mathbf{u}_{\mathbf{I}}$ represent large scale or resolved variables, which usually have only a small dimension, while the dimension of the unresolved or unobserved variables $\mathbf{u}_{\two}$ can be very large \cite{majda2014conceptual, majda2016introduction}.
Applications of the conditional Gaussian framework to highly nonlinear turbulent dynamical systems include  modelling and predicting the highly intermittent and non-Gaussian times series of the Madden-Julian oscillation (MJO) and monsoon \cite{chen2014predicting, chen2015predicting, chen2015predicting2}, filtering the stochastic skeleton model for the MJO \cite{chen2016filtering2}, and state estimation of the turbulent ocean flows from noisy Lagrangian tracers \cite{chen2014information, chen2015noisy, chen2016model}. Other studies that also fit into the conditional Gaussian framework includes the dynamic stochastic superresolution of sparsely observed turbulent systems using cheap exactly solvable forecast models \cite{branicki2013dynamic, keating2012new}, stochastic superparameterization for geophysical turbulent flows \cite{majda2014new}, physics constrained nonlinear regression models \cite{majda2012physics, harlim2014ensemble}, stochastic parameterized extended Kalman filter \cite{gershgorin2010test, gershgorin2010improving, branicki2012quantifying, chen2014mcmc, lee2017derivation} and blended particle filters for high-dimensional chaotic systems \cite{majda2014blended}.

The  efficient statistically accurate algorithms \cite{chen2017efficient} involve a hybrid strategy that requires only a small number of  samples.   In these algorithms, a conditional Gaussian mixture in the high-dimensional subspace of $\mathbf{u}_\two$ via an extremely efficient parametric method is combined with a judicious  Gaussian  kernel density estimation in the low-dimensional subspace of $\mathbf{u}_\mathbf{I}$. In particular, the conditional Gaussian distributions in the high-dimensional subspace are solved via  closed analytical formulae and are therefore computationally efficient and accurate. The full non-Gaussian joint PDF of the system is then given by a Gaussian mixture.
One remarkable feature of these efficient hybrid algorithms is that each conditional Gaussian distribution is able to cover a significant portion of the high-dimensional PDF. This guarantees the sufficiency of using only a small number of samples, which overcomes the curse of dimensionality. It has been shown in a stringent set of numerical tests \cite{chen2017efficient} that with an order of $O(100)$ samples the mixture distribution has a significant skill in capturing both the statistically steady state and the transient behavior with fat tails of the high-dimensional non-Gaussian PDFs  in up to $6$ dimensions while an order of $O(10^6)$ samples is required in the Monte Carlo simulation to reach the same accuracy. In \cite{chen2017efficient}, the restriction to $6$ dimension of the hybrid method is not essential but was utilized to allow comprehensive validation of the statistics in the truth model with an instructive simple model.


This article serves as a rigorous analysis for these efficient statistically accurate algorithms. The main focus here is the accuracy of the recovered PDFs in terms of the sample size $L$ as well as its dependence on different factors, in particular the dimension of the state variables and the time span. Throughout the article, the mean integrated square error (MISE) is used to quantify the accuracy.


Our first result \cref{thm:MISE} reveals that the MISE in the recovered high-dimensional PDFs associated with the unresolved variables $\utwo$ is bounded by $\E (\text{det}(\bfR_\two)^{-1/2})$, where $\bfR_\two$ is the conditional covariance of $\utwo$ given the trajectory of $\uone$. Notably, $\bfR_\two$ is completely determined by the underlying dynamical systems and has no dependence on the sample size $L$.
In contrast, if a direct kernel density method is applied to recover the PDF of $\utwo$, then the bandwidth of the kernel $H$ is scaled as the reciprocal of  $L$ to a certain power in order to minimize the MISE and the resulting MISE is  proportional to $L^{-1/N_\two}$, which means $L$ has to increase exponentially with $N_\two$ to guarantee the accuracy in the solution. This indicates the curse of dimensionality in the direct kernel density estimation and other smoothed versions of Monte Carlo methods.
Such a notorious issue is overcome by the efficient statistically accurate algorithms due to the independence between $\bfR_\two$ and $L$ in the high-dimensional subspace of $\utwo$.
Another significant feature of the efficient statistically accurate algorithms is their long term persistence, which is affirmed by  \cref{thm:ergodic} in a rigorous way provided that the joint process $(\mathbf{u}_\mathbf{I},\mathbf{u}_\two)$ is controllable and stochastically stable.  \cref{thm:ergodic} also supplies a lower bound of $\mathbf{R}_\two$ using the controllability condition. In addition, \cref{prop:SF} demonstrates that  dynamical systems with  energy conserving quadratic nonlinear interactions as in most geophysical and engineering turbulence \cite{majda2016introduction} automatically satisfy all the conditions for the long time persistence, which justifies the skillful performance of the efficient statistically accurate algorithms in the numerical tests reported in \cite{chen2017efficient}.  Further validations of the controllability and other theoretical conditions in the algorithms are demonstrated in the numerical simulations at the end of this article.

The remaining of this article is organized as follows. The high-dimensional nonlinear turbulent dynamical systems with conditional Gaussian structures are summarized in \cref{Sec:Systems}, which is followed by a brief review of the efficient statistically accurate algorithms in \cite{chen2017efficient} for solving the PDFs of such kind of systems. The main theoretical results are shown in \cref{Sec:Theorems}, where the proofs are included in \cref{Sec:Proofs} and the appendix. In \cref{Sec:Numerics}, numerical tests on a nonlinear triad model and its modified versions are used to validate the theoretical results. Conclusion and discussions are given in \cref{Sec:Conclusion}.

\section{Review of the efficient statistically accurate algorithms for solving the PDFs of nonlinear dynamical systems with conditional Gaussian structures}\label{Sec:Systems}
\subsection{High-dimensional conditional Gaussian models with nonlinear and intermittent dynamical features }
The general framework of high-dimensional conditional Gaussian models is given as follows \cite{liptser2001statistics, chen2016filtering}:
\begin{subequations}\label{Conditional_Gaussian_System}
\begin{align}
    d\mathbf{u}_{\mathbf{I}} &= [\mathbf{A}_0(t,\mathbf{u}_{\mathbf{I}})+\mathbf{A}_1(t,\mathbf{u}_{\mathbf{I}})\mathbf{u}_{\two}]dt + \boldsymbol{\Sigma}_{\mathbf{I}}(t,\mathbf{u}_{\mathbf{I}})d\mathbf{W}_{\mathbf{I}}(t),\label{Conditional_Gaussian_System1}\\
    d\mathbf{u}_{\two} &= [\mathbf{a}_0(t,\mathbf{u}_{\mathbf{I}})+\mathbf{a}_1(t,\mathbf{u}_{\mathbf{I}})\mathbf{u}_{\two}]dt + \boldsymbol{\Sigma}_{\two}(t,\mathbf{u}_{\mathbf{I}})d\mathbf{W}_{\two}(t), \label{Conditional_Gaussian_System2}
\end{align}
\end{subequations}
where the state variables are $\mathbf{u} = (\mathbf{u}_{\mathbf{I}}, \mathbf{u}_{\two})$ with both $\mathbf{u}_{\mathbf{I}}\in R^{N_\mathbf{I}}$ and $\mathbf{u}_{\two}\in R^{N_\two}$ being multidimensional variables. In \eqref{Conditional_Gaussian_System}, $\mathbf{A}_0, \mathbf{A}_1, \mathbf{a}_0, \mathbf{a}_1, \boldsymbol{\Sigma}_{\mathbf{I}}$ and $\boldsymbol{\Sigma}_{\two}$ are vectors and matrices that are functions of  time $t$ and the state variables $\mathbf{u}_{\mathbf{I}}$, and $\mathbf{W}_{\mathbf{I}}(t)$ and $\mathbf{W}_{\two}(t)$ are independent Wiener processes. Here the noise coefficient matrix $\boldsymbol{\Sigma}_{\mathbf{I}}$ is non-degenerated in order to guarantee the observability while there is no special requirement for $\boldsymbol{\Sigma}_{\two}$.
The dynamics \eqref{Conditional_Gaussian_System} are named as conditional Gaussian systems due to the fact that once a single trajectory $\mathbf{u}_{\mathbf{I}}(s)$ for $s\leq t$ is given, $\mathbf{u}_{\two}(t)$ conditioned on $\mathbf{u}_{\mathbf{I}}(s)$ becomes a Gaussian process with mean $\mathbf{\bar{u}}_{\two}(t)$ and covariance $\mathbf{R}_{\two}(t)$, i.e.,
\begin{equation}\label{CG_PDF}
    p\big(\mathbf{u}_{\two}(t)|\mathbf{u}_{\mathbf{I}}(s\leq t)\big) \sim \mathcal{N}(\mathbf{\bar{u}}_{\two}(t), \mathbf{R}_{\two}(t)).
\end{equation}
Despite the conditional Gaussianity, the coupled system \eqref{Conditional_Gaussian_System} remains highly nonlinear and is able to capture the strong non-Gaussian features as observed in nature \cite{chen2016filtering}. One of the desirable properties of the conditional Gaussian system \eqref{Conditional_Gaussian_System} is that the conditional distribution in \eqref{CG_PDF} has the following closed analytical form \cite{liptser2001statistics},
\begin{equation}\label{CG_Result}
\begin{aligned}
        d\mathbf{\bar{u}}_{\two}(t) = & [\mathbf{a}_0(t,\mathbf{u}_{\mathbf{I}})+\mathbf{a}_1(t,\mathbf{u}_{\mathbf{I}})\mathbf{\bar{u}}_{\two}]dt+(\mathbf{R}_{\two}\mathbf{A}^{*}_1(t,\mathbf{u}_{\mathbf{I}}))
        (\boldsymbol{\Sigma}_{\mathbf{I}}\boldsymbol{\Sigma}_{\mathbf{I}}^*)^{-1}(t,\mathbf{u}_{\mathbf{I}})\times\\
        &\qquad\qquad\qquad\qquad
    [d\mathbf{u}_{\mathbf{I}}-(\mathbf{A}_0(t,\mathbf{u}_{\mathbf{I}})+\mathbf{A}_1(t,\mathbf{u}_{\mathbf{I}})\mathbf{\bar{u}}_{\two})dt],\\
    d\mathbf{R}_{\two}(t) = & \left\{\mathbf{a}_1(t,\mathbf{u}_{\mathbf{I}})\mathbf{R}_{\two}+\mathbf{R}_{\two}\mathbf{a}^{*}_1(t,\mathbf{u}_{\mathbf{I}})+(\boldsymbol{\Sigma}_{\two}\boldsymbol{\Sigma}_{\two}^*)(t,\mathbf{u}_{\mathbf{I}})\right.\\
    &\qquad\qquad\left.-(\mathbf{R}_{\two}\mathbf{A}^{*}_1(t,\mathbf{u}_{\mathbf{I}}))
    (\boldsymbol{\Sigma}_{\mathbf{I}}\boldsymbol{\Sigma}_{\mathbf{I}}^*)^{-1}(t,\mathbf{u}_{\mathbf{I}})(\mathbf{R}_{\two}\mathbf{A}^{*}_1(t,\mathbf{u}_{\mathbf{I}}))^*\right\}dt.
\end{aligned}
\end{equation}

In most geophysical and engineering turbulent dynamical systems, the nonlinear terms such as the nonlinear advection have  quadratic forms and these quadratic nonlinear interactions conserve energy \cite{harlim2014ensemble, majda2015statistical, majda2012physics, majda2016introduction, majda1999models, majda2001mathematical}. The nonlinear interactions allow  energy transfer between different scales that induces  intermittent instabilities in the turbulent dynamical systems. Such  instabilities are then mitigated by energy-conserving quadratic nonlinear interactions that transfer energy back to the linearly stable modes where it is dissipated, resulting in a statistical steady state. Note that the nonlinear turbulent systems without the energy-conserving nonlinear interactions may suffer from non-physical finite-time blow up of statistical solutions and pathological behavior of the related invariant measure \cite{majda2012fundamental}. Mathematically, the turbulent dynamical systems with energy-conserving quadratic nonlinear interactions have the following abstract forms:
\begin{equation}\label{EnergyConserveModel}
  d\mathbf{u} = \big[ -\Lambda\mathbf{u} + \mathbf{B}(\mathbf{u},\mathbf{u}) + \mathbf{F}(t) \big] dt + \boldsymbol{\Sigma}(t,\mathbf{u})d\mathbf{W}(t),
\end{equation}
where $-\Lambda=\mathbf{L}+\mathbf{D}$. Here, $\mathbf{L}$ is a skew-symmetric linear operator that can represent the $\beta$ effect of Earth's curvature and topography, while $\mathbf{D}$ is a negative definite symmetric operator representing dissipative processes such as surface drag, radiative damping and viscosity, etc \cite{salmon1998lectures, thompson2006scaling, majda2006nonlinear, vallis2017atmospheric}. The quadratic operator $\mathbf{B}(\mathbf{u},\mathbf{u})$ conserves  energy by itself so that it satisfies the following:
\begin{equation}\label{EnergyConserving}
  \mathbf{u}\cdot\mathbf{B}(\mathbf{u},\mathbf{u})  = 0.
\end{equation}
Notably, a rich class of turbulent models with energy-conserving quadratic nonlinear interactions in \eqref{EnergyConserveModel} belongs to the conditional Gaussian systems \eqref{Conditional_Gaussian_System}, including the noisy version of Lorenz 63 model \cite{lorenz1963deterministic}, the reduced stochastic climate model \cite{majda2008applied, majda2005information}, the nonlinear triad model mimicking structural features of low-frequency variability of GCMs with non-Gaussian features \cite{majda2009normal}, the modified conceptual dynamical model for turbulence \cite{majda2014conceptual}, and the two-layer Lorenz 96 model \cite{lee2017multiscale}. See \cite{chen2017efficient} and its appendix for a general framework of conditional Gaussian systems with energy-conserving nonlinear interactions as well as concrete examples.

\subsection{The efficient statistically accurate algorithms for solving the PDFs of the conditional Gaussian systems}
Assume the dimension $N_\mathbf{I}$ of the observed variables is low, while the dimension $N_\two$ of the unobserved variables can be high. This is the typical scenario in most turbulent dynamical systems, where the low-dimensional variables $\mathbf{u}_{\mathbf{I}}$ represent large scales or resolved variables while the high-dimensional ones $\mathbf{u}_{\two}$ stand for the unresolved and unobserved variables \cite{majda2014conceptual, majda2016introduction}.

Below, we summarize the procedures of the efficient statistical algorithms developed in \cite{chen2017efficient}. First, we generate $L$ independent trajectories from the stochastic dynamical systems \eqref{Conditional_Gaussian_System}. 
In fact,
the only information that is required for these algorithms is $L$ independent trajectories of the observed variables, namely $\mathbf{u}^{1}_\mathbf{I}(s\leq t),\ldots,\mathbf{u}^{L}_\mathbf{I}(s\leq t)$. 
Then, different strategies are used to deal with the observed variables $\mathbf{u}_{\mathbf{I}}$ and unobserved variables $\mathbf{u}_{\two}$, respectively.
The PDF of $\mathbf{u}_{\two}$ is estimated via a parametric method that exploits the closed form of the conditional Gaussian posterior statistics \eqref{CG_Result},
\begin{equation}\label{PDF_U_II}
   p(\mathbf{u}_{\two}(t)) = \lim_{L\to\infty}\frac{1}{L}\sum_{i=1}^Lp(\mathbf{u}_{\two}(t)|\mathbf{u}^i_{\mathbf{I}}(s\leq t)).
\end{equation}
Note that the limit $L\to\infty$ in \eqref{PDF_U_II} (as well as \eqref{PDF_U_I} and \eqref{Joint0} below) is taken 
to illustrate the statistical intuition, while the estimator is the non-asymptotic version. 
On the other hand,  a  Gaussian kernel density estimation method is used for solving the PDF of the observed variables $\mathbf{u}_{\mathbf{I}}$,
\begin{equation}\label{PDF_U_I}
   p\big(\mathbf{u}_{\mathbf{I}}(t)\big)= \lim_{L\to\infty}\frac{1}{L}\sum_{i=1}^L K_\mathbf{H}\Big(\mathbf{u}_{\mathbf{I}}(t)-\mathbf{u}^i_{\mathbf{I}}(t)\Big),
\end{equation}
where $\mathbf{H}=\mathbf{H}(t)$ is the bandwidth matrix, and $K_\mathbf{H}(\cdot)$ is a Gaussian kernel centered at each sample point with covariance $\mathbf{H}(t)$,
\begin{equation}\label{U_I_Mixture}
  K_\mathbf{H}\Big(\mathbf{u}_{\mathbf{I}}(t)-\mathbf{u}^i_{\mathbf{I}}(t)\Big)\sim \mathcal{N}\Big(\mathbf{u}^i_{\mathbf{I}}(t), \mathbf{H}(t)\Big).
\end{equation}
Below, we simply use $\mathbf{H}$ to represent the bandwidth at time $t$ for the notation simplicity.

The kernel density estimation algorithm here involves a ``solve-the-equation plug-in'' approach for optimizing the bandwidth, the idea of which was originally proposed in  \cite{botev2010kernel}. The solve-the-equation approach does not impose any requirement for the profile of the underlying PDF. Therefore, it works for the non-Gaussian cases and the
computational cost comes from numerically solving a scalar high order algebraic equation for the optimal bandwidth in order to minimize the asymptotic mean integrated squared error (AMISE) in the estimator. Furthermore, we adopt a diagonal matrix for $\mathbf{H}$. This greatly reduces the computational costs while remains the results with reasonable accuracy. Note that in the limit $L\to\infty$, the kernel density method is simply the Monte Carlo simulation, where the bandwidth shrinks to zero.

Finally, with \eqref{PDF_U_II} and \eqref{PDF_U_I} in hand, a hybrid method is applied to solve the joint PDF of $\mathbf{u}_{\mathbf{I}}$ and $\mathbf{u}_{\two}$  through a Gaussian mixture,
\begin{equation}\label{Joint0}
    p(\mathbf{u}_{\mathbf{I}}(t),\mathbf{u}_{\two}(t)) = \lim_{L\to\infty} \frac{1}{L}\sum_{i=1}^L \Big(K_\mathbf{H}(\mathbf{u}_{\mathbf{I}}(t)-\mathbf{u}_{\mathbf{I}}^i(t))\cdot p(\mathbf{u}_{\two}(t)|\mathbf{u}^i_{\mathbf{I}}(s\leq t))\Big).
\end{equation}
One important features of these algorithms is that the solutions of both the two marginal distributions in \eqref{PDF_U_II} and \eqref{PDF_U_I} and the joint distribution in \eqref{Joint0} are consistent with those of solving the Fokker-Planck equation for $p(\mathbf{u}_\two(t)), p(\mathbf{u}_\mathbf{I}(t))$ and $p(\mathbf{u}_{\mathbf{I}}(t),\mathbf{u}_{\two}(t))$, respectively.

Practically, $L\sim O(100)$ is sufficient for the efficient hybrid method \eqref{Joint0} to solve the joint PDF with $N_\one\leq 3$ and $N_\two\sim 10$ while an order of $O(10^6)$ samples is required for solving the joint PDF using classical Monte Carlo methods to reach the same accuracy for a $6$ dimensional turbulent system \cite{chen2017efficient}.
Since $L$ is only of order $O(100)$, the $L$ independent trajectories  $\mathbf{u}^{1}_\mathbf{I}(s\leq t),\ldots,\mathbf{u}^{L}_\mathbf{I}(s\leq t)$  can be obtained by running a Monte Carlo simulation for the coupled system \eqref{Conditional_Gaussian_System} with $L$ samples, which is computationally affordable. In addition, the closed form of the $L$ conditional distributions in \eqref{PDF_U_II} can be computed in a parallel way due to their independence, which further reduces the computational cost. See \cite{chen2017efficient} for more details.

\section{Main theoretical results}\label{Sec:Theorems}
The rigorous analysis of the efficient statistically accurate algorithms involving the hybrid strategy \eqref{Joint0} is studied in this section. For comparison, the theoretical results by applying the kernel density estimation method to the full system \eqref{Conditional_Gaussian_System} is also illustrated. Note that the kernel density estimation is essentially the Monte Carlo simulation when $L$ is large and therefore it suffers from the curse of dimensionality. Such comparison facilitates the understanding of the advantages of the efficient algorithm \eqref{Joint0} in recovering the high-dimensional subspace of $\mathbf{u}_\two$ using only a small number of samples. Below,  ${p}_t(\uone,\utwo)$ represents the true PDF while $\tilde{p}_t(\uone,\utwo)$ and $\hat{p}_t(\uone,\utwo)$ stand for the recovered PDFs based on the pure kernel density estimation and the efficient hybrid method \eqref{Joint0}, respectively. \medskip

\noindent\emph{Kernel density estimation for the joint PDF.}
\begin{align}
\tilde{p}_t(\uone,\utwo)&=\frac{1}{L}\sum_{i=1}^L K_H ((\uone,\utwo)- (\uone^i(t), \utwo^{i}(t))),\label{eqn:directkernel}\\
\mbox{with~~}K_H(\uone,\utwo)&= (2\pi H)^{-\frac{N_\one+N_\two}{2}} \exp\left(-\frac1{2H} \sum_{i=1}^{N_\one} c^2_i\bfu_{\one,i}^2-\frac1{2H} \sum_{i=1}^{N_\two} c^2_{i+N_\one}\bfu_{\two,i}^2\right).\label{eqn:kernel}
\end{align}
\noindent\emph{Hybrid method --- kernel density estimation for $\mathbf{u}_\mathbf{I}$ and conditional Gaussian mixture for $\mathbf{u}_\two$.}
\begin{align}
\hat{p}_t(\uone,\utwo)&= \frac{1}{L}\sum_{i=1}^L K_H (\uone- \uone^{i}(t))p(\utwo| \uone^{i}(s\leq t)).\label{eqn:condGauss}\\
\mbox{with~~}K_H(\uone)&= (2\pi H)^{-\frac{N_\one}{2}} \exp\left(-\frac1{2H} \sum_{i=1}^{N_\one} c^2_i \bfu_{\one,i}^2\right).\label{eqn:kernel2}
\end{align}
In \eqref{eqn:kernel} and \eqref{eqn:kernel2}, we let $\mathbf{H}=H\mathbf{C}$ as in \eqref{Joint0}. The scalar $H$ is the scale of the bandwidth  \cite{silverman1981using, wand1992error, wand1994multivariate, botev2010kernel} and  $c^2_i$ are the diagonal terms of $\bfC$ such that $c^2_iH$ represents the bandwidth in one direction.
In the following, we mostly concern the performance of $\tilde{p}_t$ and $\hat{p}_t$ when $L$ is large.


One standard metric to measure the performance of a density estimator is the
 mean integrated squared error (MISE). The MISE of the hybrid method, for example, is the average $L^2$ distance to the true density:
\[
 \text{MISE}=\E\int |p_t(\uone,\utwo)-\hat{p}_t(\uone,\utwo)|^2d\uone d\utwo.
\]
Note that $\hat{p}_t$ relies on the realization of the samples and therefore it is natural to take the expectation of the distance.

Applying the Bias-Variance decomposition \cite{friedman1997bias} to the MISE yields
\begin{equation}
\label{eqn:MISEdec}
\text{MISE}=\underbrace{\E \int |\hat{p}_t(\uone,\utwo)- \bar{p}_t(\uone,\utwo)|^2 d\uone d\utwo}_{\mbox{Bias}}+\underbrace{\int |p_t(\uone,\utwo)- \bar{p}_t(\uone,\utwo)|^2 d\uone d\utwo}_{\mbox{Variance}},
\end{equation}
where $\bar{p}_t:=\E \hat{p}_t$. The variance part comes from the sampling error of the method and the bias part comes from the usage of the kernel method. See \eqref{eqn:MISEdecomp} for a direct proof of this decomposition.

The MISE and its decomposition \eqref{eqn:MISEdec} will be used to understand the performance of the two density estimation methods in \eqref{eqn:directkernel} and \eqref{eqn:condGauss}, where the scenarios with a large number of samples and a large dimension of the variables $N_\two$ are of particular interest. Main results are presented below and the rigorous proofs of these results are shown in \cref{Sec:Proofs}. Note that despite quite a few studies of kernel density estimation, especially in the asymptotic limit, exist in literature  \cite{silverman1981using, wand1992error, wand1994multivariate, jones1996brief, botev2010kernel}, no analysis has been established for the hybrid method \eqref{eqn:condGauss}. Moreover, the results here are all non-asymptotic, and therefore they hold for arbitrary choice of bandwidth parameters. This is important in practice, as the bandwidth matrix $\mathbf{H}(t)$ may change with $t$.

\subsection{MISE of the hybrid method}
The main result of our analysis is the following:
\label{sec:joint}
\begin{thm}
\label{thm:MISE}
The two parts of MISE in \eqref{eqn:MISEdec} for the hybrid method \eqref{eqn:condGauss} are bounded:
\begin{equation}
\label{eqn:MISEbound}
\begin{gathered}
\hat{p}_t \,\,\mbox{Variance}\leq \frac1L\E \left(\prod_{i=1}^{N_\one} (\pi Hc_i^2) \text{det}(\pi \bfR_\two(t)) \right)^{-\frac12},\\
\hat{p}_t \,\,\mbox{Bias}\leq \frac{1+\delta}{4} H^2 J\left( \sum_{i=1}^{N_\one} c_i^2\partial^2_{\bfu_{\one,i}^2} p_t(\uone,\utwo)\right)+\frac{1+\delta^{-1}}2 M^2 H^3\left(\sum_{i=1}^{N_\one}c_i^2\right)^3 J(M(\uone,\utwo)).
\end{gathered}
\end{equation}
Here $\delta$ is any fixed strictly positive number. $\E$ is the statistical average.  $J(f(\uone,\utwo))$  denotes the integral $\int f^2(\uone,\utwo) d\uone d\utwo$.
The function $M(\uone,\utwo)$ is an upper bound of the third order directional derivative of $p_t$ in the direction of $\uone$ around $(\uone,\utwo)$.  That is, we assume
\begin{equation}
\label{eqn:thirder}
\left| \frac{d^3}{ds^3}p_t(\uone+s\bfv, \utwo)\right|\leq M(\uone,\utwo),\quad \text{for all } \bfv\in \reals^{N_\one}, |\bfv|\leq 1.
\end{equation}
\end{thm}
In a practical scenario,  as the sample size $L$ increases,  bandwidth $H$ can decrease, so that both the variance and bias terms decrease to zero. By taking $\delta$ close to zero and ignoring the higher order term in the bias upper bound, we recover an upper bound similar to the asymptotic MISE (AMISE)  in \cite[Eqn. (2.6)]{wand1992error}, except that our method also consists a random component of $\bfR_\two(t)$:
\begin{equation}\label{eqn:AMISE}
\text{AMISE}\leq \frac1L\E \left(\prod_{i=1}^{N_\one} (\pi Hc_i^2) \text{det}(\pi \bfR_\two(t)) \right)^{-\frac12}+\frac{1}{4} H^2 J\left( \sum_{i=1}^{N_\one} c_i^2\partial^2_{\bfu_{\one,i}^2} p_t(\uone,\utwo)\right),
\end{equation}
where the two terms on the right hand side represents the variance and bias, respectively.
It is natural to equate the order of these two terms, that is letting $L H^{-\frac12N_\one}\sim O(H^2)$.
This leads to the common choice of the bandwidth \cite{jones1996brief}
\begin{equation}
\label{eqn:MISEorder}
H\sim O\left(L^{-\frac{2}{4+N_\one}}\right)\quad \text{and consequentially}\quad \text{MISE}\sim O\left(L^{-\frac{4}{4+N_\one}}\right).
\end{equation}

Notably, the variance part of MISE in \eqref{eqn:AMISE} depends on $\utwo$  through $\E \sqrt{\text{det} (\pi\bfR_\two(t))}^{-1}$, which indicates that the hybrid method in \eqref{eqn:condGauss} performs better with a larger $\bfR_\two(t)$. This is consistent with the intuition that a large $\bfR_\two(t)$ corresponds to a conditional distribution $\mathcal{N}(\barut(t), \bfR_\two(t))$ with a wide band that is able to recover a sufficient portion of the PDF.


\subsection{Comparison between the two density estimators}
\label{sec:performance}
\cref{thm:MISE} already reveals the advantage of the hybrid method \eqref{eqn:condGauss} over the the direct kernel density method  \eqref{eqn:directkernel}. For a qualitative comparison of the two methods, we can view the latter as a trivial application of the hybrid method by taking $\uone'=(\uone,\utwo)$ and  $\utwo'=\emptyset
$, and therefore $\utwo'$ is trivially linear  conditioned on $\uone'$. A direct application of \cref{thm:MISE} leads to
\begin{equation}
\label{eqn:MISEkernel}
\begin{gathered}
\tilde{p}_t \,\,\text{Variance}\leq \frac1L\E \left(\prod_{i=1}^{N_\one+N_\two} \pi Hc_i^2  \right)^{-\frac12},\\
\tilde{p}_t\,\,\text{Bias}\leq \frac{(1+\delta)H^2}4 J\left( \sum_{i=1}^{N_\one} c_i^2\partial^2_{\bfu_{\one,i}^2} p_t+\sum_{i=1}^{N_\two} c_{i+N_\one}^2\partial^2_{\bfu_{\two,i}^2} p_t\right)+\frac{(1+\delta^{-1})H^3}2  \left(\sum_{i=1}^{N_\one}c_i^2\right)^3J(\widetilde{M}).
\end{gathered}
\end{equation}
where $\widetilde{M}\geq M$ is the upper bound for third order directional derivative in $\reals^{N_\one+N_\two}$ of $p_t$. Similar results in the asymptotic setting can be found in \cite{wand1992error}.

 If we use the same bandwidth $H$ and sample size $L$ in both method, Comparing \eqref{eqn:MISEkernel} with \eqref{eqn:MISEbound}, we find that
\[
\tilde{p}_t\,\,\text{Bias bound}\geq \hat{p}_t\,\,\text{Bias bound},
\]
and moreover
\[
\frac{\tilde{p}_t\,\,\text{Variance bound}}{\hat{p}_t\,\,\text{Variance bound}}=\frac{H^{-\frac{N_\two}{2}}\prod_{i=1}^{N_\two} c_{i+N_\one}}{\E \sqrt{\text{det}(\bfR_\two(t))}^{-1}}.
\]
Practically, a large $L$ is chosen to guarantee the accuracy of the recovered PDFs, which corresponds to a small bandwidth $H$. Then the variance part of the direct kernel method is several magnitudes larger than that of hybrid method, especially when the dimension $N_\two$ is high.

As discussed above, one would optimize the choice of $H$ such that the two quantities in \eqref{eqn:MISEkernel} are of the same order, which leads to the scaling $H\sim O\left(L^{-\frac{2}{4+N_\one+N_\two}}\right)$, and also the overall $\text{MISE}\sim O\left(L^{-\frac{4}{4+N_\one+N_\two}}\right)$, However, This is much worse than the MISE associated with the conditional Gaussian method  \eqref{eqn:MISEorder} when $N_\two$ is large.  Alternatively,  if one wants the performance of the direct kernel method to be the same as the conditional Gaussian one \eqref{eqn:MISEorder}, then the sample size needs to increase to $\widetilde{L}=L^{\frac{4+N_\one+N_\two}{4+N_\one}}$, which can be many magnitudes larger than $L$.

In conclusion, direct application of the kernel method suffers from the curse of dimensionality. This is due to the fact that the variance scales with the bandwidth as $H^{-\frac{N_\one+N_\two}{2}}$, and therefore one needs to increase sample size exponentially with the dimension in order to have a small bandwidth that guarantees the accuracy of the recovered PDFs. However, when $H$ is small, the kernel density method approximates the standard Monte Carlo simulation, which suffers from the curse of dimensionality. On the other hand, the hybrid method resolves this issue by estimating the $\utwo$ part using a parametric method where the bandwidth (or the covariance) does not depend on $L$. Therefore, the performance of the hybrid method \eqref{eqn:condGauss} can be much superior than the direct kernel method \eqref{eqn:directkernel} when $N_\two$ is large.

\subsection{Marginal distribution of $\utwo(t)$}
There are scenarios where the focus is only on estimating the density of $\utwo(t)$. Again, both methods can be applied here. The direct kernel method \eqref{eqn:directkernel} results in the estimation of the marginal density
\begin{equation}\label{eqn:marginalkernel}
\tilde{p}_t(\utwo):=\frac1L\sum_{i=1}^L K_H(\utwo-\utwo^{i}(t)),\quad K_H(\utwo)= (2\pi H)^{-\frac{N_\two}{2}} \exp\left(-\frac1{2H} \sum_{i=1}^{N_\two} c^2_{i+N_\one} \bfu_{\two,i}^2\right).
\end{equation}
On the other hand, the hybrid method \eqref{eqn:condGauss} simply becomes a conditional Gaussian mixture method which contains no kernel density estimation
\begin{equation}\label{eqn:marginalcondgauss}
\hat{p}_t(\utwo):=\frac1L\sum_{i=1}^Lp(\utwo| \uone^{i}(s\leq t)).
\end{equation}
It is straightforward to check these density estimators are the marginal PDFs of the joint distributions in \eqref{eqn:directkernel} and \eqref{eqn:condGauss}.

Since there is no kernel involved for the conditional Gaussian method in \eqref{eqn:marginalcondgauss}, the MISE has a simple bound without the bias part:
\begin{prop}
\label{prop:marginal}
The marginal MISE of the conditional Gaussian estimator in \eqref{eqn:marginalcondgauss} is bounded as
\begin{equation}
\label{eqn:marginalMISE}
\hat{p}_t \,\,\text{MISE} \leq \frac{1}{L}\E \left(\text{det}(\pi \bfR_\two(t))\right)^{-\frac12}.
\end{equation}
\end{prop}
Following the derivation of \eqref{eqn:MISEkernel},  the MISE of the direct kernel method in \eqref{eqn:marginalkernel} is given by
\begin{multline*}
\tilde{p}_t\,\,\text{MISE}\leq \frac1L\E \left(\prod_{i=1}^{N_\two} \pi Hc_{i+N_\one}^2  \right)^{-\frac12}+\frac{(1+\delta)H^2}4 J\left(\sum_{i=1}^{N_\two} c_{i+N_\one}^2\partial^2_{\bfu_{\two,i}^2} p_t\right)\\+\frac{(1+\delta^{-1})H^3}2  \left(\sum_{i=1}^{N_\one}c_i^2\right)^3J(\widetilde{M}).
\end{multline*}
With the optimal choice $H\sim O\left(L^{-\frac{2}{4+N_\two}}\right)$, the direct kernel method $\text{MISE}\sim O\left(L^{-\frac{4}{4+N_\two}}\right).$ The hybrid method with the conditional Gaussian mixture is clearly superior for marginal density estimation,  as its MISE \eqref{eqn:marginalMISE} is essentially $O(L^{-1})$, and the bandwidth $H$ has no dependence on $L$.

\subsection{Fixed subspace}
\label{sec:subspace}
In many scenarios, only a part of $\utwo$ is of practical interest. To this end, we consider here $\utwo^P=\bfP \utwo$, where  $\bfP:\reals^{N_\two}\mapsto \reals^{N_\two^P}$  maps  $\utwo$ onto a lower dimensional subspace. Below, we study the estimation of the density $p^P_t(\uone,\utwo^P)$ of $(\uone(t), \utwo^P(t))$ using the hybrid method.

It is straightforward to show the conditional distribution of $\utwo^P(t)$ given $\uone(s\leq t)$ follows the Gaussian density $p (\utwo^P |\uone(s \leq t))$ of the following form
\[
\text{det}(2\pi \bfP\bfR_\two(t)\bfP^*)^{-\frac12} \exp\left(-\tfrac12 (\utwo^P-\bfP\barut(t))^*[\bfP\bfR_\two(t)\bfP^*]^{-1} (\utwo^P-\bfP\barut(t)) \right).
\]
The density of $(\uone(t),\utwo^P(t))$ can be estimated by
\[
 \hat{p}^P_t(\uone,\utwo^P)= \frac{1}{L}\sum_{i=1}^L K_H (\uone- \uone^{i}(t))p(\utwo^P |\uone^{i}(s\leq t)).
\]
Following \cref{thm:MISE}, we can show that

\begin{cor}
\label{cor:marginal}
Under the same assumption as in \cref{thm:MISE}, the MISE decomposition of  $\hat{p}^P_t$ has the following two bounds
\[
\begin{gathered}
\hat{p}^P_t\text{  Variance}\leq \frac1L\E \left(\prod_{i=1}^{N_\one} (\pi Hc_i^2) \text{det}(\pi \bfP\bfR_\two(t)\bfP^*) \right)^{-\frac12},\\
\hat{p}^P_t\text{  Bias}\leq \frac{1+\delta}{4} H^2 J\left( \sum_{i=1}^{N_\one} c_i^2\partial^2_{\bfu_{\one,i}^2} p^P_t(\uone,\utwo^P)\right)+\frac{1+\delta^{-1}}2  H^3\left(\sum_{i=1}^{N_\one}c_i^2\right)^3 J(M^P(\uone,\utwo^P)),
\end{gathered}
\]
where $M^P$ is a upper bound of third order derivative of $p^P_t$ in $\uone$, as in \eqref{eqn:thirder}.
\end{cor}
Notably, the variance term depends only on $\E \sqrt{\text{det}(\pi \bfP \bfR_\two(t) \bfP^*)}^{-1}$, where $\bfP\bfR_\two(t)\bfP^*$ is a $N_\two^P\times N_\two^P$ matrix that is independent of the components complementary to $\utwo^P(t)$. In other words,  the performance of the hybrid estimator  on a certain part of the components  is independent of the other components. This is particularly useful when $N^P_\two$ is small. Note that such a property also holds for the direct kernel method but in practice the kernel method works only for the case when $N_\two$ is small.

\subsection{Controllability and a lower bound of $\bfR_\two$}
\label{sec:lowerbound}
According to \cref{thm:MISE},  $\bfR_\two(t)$ controls the sampling variance term in the MISE. Therefore, it is desirable to derive  a lower bound for $\bfR_\two(t)$.  Note that in the conditional Gaussian system \eqref{Conditional_Gaussian_System},  $\uone$ can be interpreted as an observation of $\utwo$, and $p(\utwo|\uone(s\leq t))$ is essentially the optimal Kalman filter with covariance $\bfR_\two(t)$. Therefore, a lower bound of $\bfR_\two(t)$ can be guaranteed by the controllability of the associated signal-observation system. In short, the controllability condition ensures the noise in the system is regular enough such that the optimal filter is not accurate to a singular degree in any component. More discussions on the controllability of Kalman filters can be found in \cite{chui1999kalman, crisan2011oxford, majda2012filtering}. A recent work \cite{bishop2016stability} has summarized some of the major results in this area. It is noteworthy that since the term $\bfa_1$ depends on realization of $\uone$, both the controllability  condition and  the lower  bounds rely on the realization of $\uone$.

In our context, a  standard way to characterize this notion is the following assumption:
\begin{aspt}
\label{aspt:control}
Let $\calE_{s,t}$ be the matrix flow generated by $\bfa_1$:
\[
\frac{d}{dt}\calE_{s,t}=\bfa_1(t,\uone(t))\calE_{s,t},\quad \calE_{s,s}=I_{N_\two}.
\]
Suppose there are constants $v>0,m\geq 0$ and  $D_c\geq 1$ such that for any $t\geq v$ and $s\in [t-v,t]$,
\[
 D^{-1}_c I_{N_\two}\preceq\calE_{s,t} \calE_{s,t}^*\preceq D_c I_{N_\two},\quad \sigma_{\two,-}^2 I_{N_\two} \preceq\bfSigma_\two^* \bfSigma_\two\preceq \sigma_{\two,+}^2 I_{N_\two},
\]
\[
\bfA^*_1(t,\uone(t)) [\bfSigma_\one\bfSigma^*_\one]^{-1}\bfA_1(t,\uone(t))\preceq D_c (|\uone(t)|^{2m} +1)I_{N_\two}.
\]
Throughout this paper, for two real symmetric matrices $A$ and $B$, we use $A\preceq B$ to indicate that $B-A$ is a positive semi-definite matrix.

\end{aspt}
While $\bfA_1^* (\bfSigma_\one \bfSigma_\one^*)^{-1} \bfA_1$ actually concerns of observability, this bound is very mild. Thus, we still call \cref{aspt:control} the controllability condition.

\begin{prop}
\label{prop:control}
Suppose $N_\two\geq 2$, and the controllability condition, \cref{aspt:control} holds, then for any $t\geq v$,
 $\bfR_\two(t)\succeq h_{t,v}^{-1}(\uone) I_{N_\two}$, where
 \[
h_{t,v}(\uone):=v^2 \sigma^2_{\two, +} \sigma^{-2}_{\two, -}D^6_c\left(v+\int^t_{t-v} |\uone(r)|^{2m}dr\right) +v^{-1}D_c\sigma^{-2}_{\two, -}.
\]
In particular there are constants $D_1$ and $D_2$ such that
\[
\E\sqrt{\text{det}\bfR_\two(t)}^{-1}\leq D_1+D_2\int^t_{t-v} \E |\uone(r)|^{mN_\two}dr.
\]
\end{prop}
The dependence of $\bfR_\two(t)$ on $\uone( s)|_{t-v\leq s\leq t}$ comes from the observational term $\bfA_1$. As is seen from \eqref{CG_Result}, if $\bfA_1^* (\bfSigma_\one \bfSigma_\one^*)^{-1} \bfA_1$ is large, $\bfR_\two(t)$ has a large quadratic damping, which can bring it to a very low level.

In symmetry, an upper bound can be derived if  a lower bound of $\bfA_1^* (\bfSigma_\one \bfSigma_\one^*)^{-1} \bfA_1$ is assumed. Furthermore, one can show that the Riccati flow of $\bfR_\two(t)$ is contractive, so its dependence on $\bfR_\two(0)$ is diminishing. Since these results are not directly related to the performance of the hybrid estimator, we put them in  the appendix along with the verification of Proposition \ref{prop:control}.

\subsection{Long time performance}
The simulation of $(\uone^{i}(t), \utwo^{i}(t))$ can be maintained continuously, and  the conditional Gaussian density estimator \eqref{eqn:condGauss} can be applied for an online estimation. One important question to ask is whether the performance, and in particular the MISE, degenerates with time. If this is the case, additional samples are needed to reinforce the estimation, which is however usually difficult to carry out in practice. In this subsection, we show that the conditional Gaussian density estimator has a long time stable performance, as long as the joint process $(\uone,\utwo)$ is stable and ergodic.

In stochastic analysis, the stability and ergodicity of a process can be guaranteed by energy dissipation and non-degenerate stochastic forcing. For our purpose, we can assume the energy is dissipative, while the noise is elliptic \cite{majda2016ergodicity}.
\begin{aspt}
\label{aspt:ergodic}
Suppose $\bfSigma_\one$ and $\bfSigma_\two$ are full rank, and the energy is dissipative with a rate $\rho>0$ and a constant $D_e$
\begin{equation}
\label{eqn:dissipative}
\uone \cdot (\bfA_0+\bfA_1\utwo) + \utwo \cdot (\bfa_0+\bfa_1\utwo) \leq  -\rho(|\uone|^2+|\utwo|^2)+D_e.
\end{equation}
\end{aspt}

\begin{thm}
\label{thm:ergodic}
Under  \cref{aspt:ergodic}, the following hold.

\noindent 1) The joint density $p_t$ converges geometrically to an ergodic measure $p_\infty$ with a rate $c>0$. In particular, there is a constant $D_0$ so that
\begin{equation}
\label{tmp:Poincare}
\int \left|\frac{p_t}{p_\infty}(\uone,\utwo)-1\right|^2 p_\infty(\uone,\utwo)d\uone d\utwo\leq D_0 e^{-ct}\langle |\bfu|^2+1, p_0\rangle \left\|\frac{p_0}{p_\infty}-1\right\|_\infty^2.
\end{equation}
Here $\langle |\bfu|^2+1, p_0\rangle$ denotes the quantity $\int (|\uone|^2+|\utwo|^2+1) p_0(\uone,\utwo) d\uone d\utwo$, and $\|f\|_\infty$ denotes the supremum $\|f\|_\infty=\sup_{\uone,\utwo} |f(\uone,\utwo)|$.

\noindent 2) Suppose \cref{aspt:control} also holds,  then for any $t>0$ and $\delta>0$, $N_\two\geq 2$, the two parts of the MISE using the hybrid method are bounded by
\begin{align*}
\hat{p}_t\text{  Variance}&\leq \frac{D_{m,N_\two,v}}{L \pi^{\frac{N_\one+N_\two}{2}}H^{\frac{N_\one}{2}}\prod_{i=1}^{N_\one}c_i}\left(\exp(-\tfrac12\rho mN_\two t) \E |\bfu(0)|^{mN_\two}+D_{m,N_\two,v}\right),\\
\hat{p}_t\text{  Bias}\leq &\frac{(1+\delta)^2}{4} H^2 J\left( \sum_{i=1}^{N_\one} c_i^2\partial^2_{\bfu_{\one,i}^2} p_\infty(\uone,\utwo)\right)+\frac{(1+\delta)^2}{2\delta}  H^3\left(\sum_{i=1}^{N_\one}c_i^2\right)^3 J(M_\infty(\uone,\utwo))\\
&+8(1+\delta^{-1})D_0 e^{-ct}\langle |\bfu|^2+1, p_0\rangle \left\|\tfrac{p_0}{p_\infty}-1\right\|_\infty^2 \|p_\infty\|_\infty,
\end{align*}
where $D_{m,N_\two,v}$ is a constant independent of $L$ and $H$, and $M_\infty$ is a  bound for the third order $\uone$-directional derivative of $p_\infty$ as in \eqref{eqn:thirder}.
\end{thm}
In particular, when $t\to \infty$, we have
\begin{align*}
\limsup_{t\to\infty}\text{MISE}&\leq \frac{D^2_{m,N_\two,v}}{L \pi^{\frac{N_\one+N_\two}{2}}H^{\frac{N_\one}{2}}\prod_{i=1}^{N_\one}c_i}+ \frac{(1+\delta)^2}{4} H^2 J\left( \sum_{i=1}^{N_\one} c_i^2\partial^2_{\bfu_{\one,i}^2} p_\infty(\uone,\utwo)\right)\\
&\quad+\frac{(1+\delta)^2}{2\delta}  H^3\left(\sum_{i=1}^{N_\one}c_i^2\right)^3 J(M_\infty(\uone,\utwo)).
\end{align*}
This leads to the same bandwidth and MISE scaling with $L$, namely:
\[
H\sim O\left(L^{-\frac{2}{4+N_\one}}\right)\quad \text{and}\quad \text{MISE}\sim O\left(L^{-\frac{4}{4+N_\one}}\right).
\]
The proof strategy of \cref{thm:ergodic} is straightforward. The first part is simply corollaries of \cite{meyn2012markov, mattingly2002ergodicity, bakry2008rate}. To reach a bound on the variance part in 2), it suffices to have a lower bound on $\E \sqrt{\text{det}\bfR_\two(t)}^{-1}$. This can be achieved by  \cref{prop:control} and an energy dissipation argument. For the bias term, we use the Poincar\'e inequality \eqref{tmp:Poincare} to approximate it with the bias term at equilibrium.

\subsection{Conditional Gaussian turbulent dynamical systems with energy-conserving \\quadratic nonlinearity}\label{sec:energyconservingmodel}
Recall the turbulence model $\bfu$ with quadratic energy conserving nonlinear interactions \eqref{EnergyConserveModel}--\eqref{EnergyConserving}
\[
d \bfu= -\Lambda \bfu dt+\bfB(\bfu,\bfu) dt+\bfF dt+\bfSigma d\bfW_t.
\]
The linear damping part provides a uniform dissipation, so for some $\lambda_->0$,
\[
\bfu \cdot \Lambda \bfu\geq \lambda_- |\bfu|^2,
\]
and the nonlinearity term $\bfB$ is quadratic and conserves energy.

In our conditional Gaussian setup, we can decompose the dynamics into the form below
\begin{equation}
\label{sys:SF}
\begin{gathered}
d\uone=(-\Lambda_{\one,0}\uone+\bfB_{\one,0} (\uone,\uone)+\bfF_\one)dt+(-\Lambda_{\one,1}+\bfB_{\one,1}(\uone))\utwo dt+\bfSigma_\one d\bfW_\one,\\
d\utwo=(-\Lambda_{\two,0}\uone+\bfB_{\two,0} (\uone,\uone)+\bfF_\two)dt+(-\Lambda_{\two,1}+\bfB_{\two,1}(\uone))\utwo dt+\bfSigma_\two d\bfW_\two.
\end{gathered}
\end{equation}
The quantities in the brackets naturally correspond to  $\bfA_0,\bfA_1,\bfa_0$ and  $\bfa_1$ respectively.

For the damping term $\Lambda$, we assume there are constants $0<\lambda_-\leq \lambda_+$,
\begin{equation}
\label{tmp:linear}
\lambda_- I_{N_\one+N_\two}\preceq\begin{bmatrix}\Lambda_{\one,0}  &\Lambda_{\one,1}\\ \Lambda_{\two,0} &\Lambda_{\two,1}\end{bmatrix}\preceq \lambda_+ I_{N_\one+N_\two}.
\end{equation}
The energy conservation condition, $\bfu\cdot \bfB(\bfu,\bfu)=0$, requires that
\begin{equation}
\label{tmp:quadratic}
\uone\cdot \bfB_{\one,0} (\uone,\uone)=0, \quad \utwo\cdot\bfB_{\two,1}(\uone)\utwo=0, \quad \uone\cdot\bfB_{\one,1}(\uone)\utwo+\utwo\cdot\bfB_{\two,0} (\uone,\uone)=0.
\end{equation}
See the Appendix of \cite{chen2017efficient} for details.
\begin{prop}
\label{prop:SF}
For the stochastic flow with energy conserving quadratic nonlinearity \eqref{sys:SF}, assume that \eqref{tmp:linear}  and \eqref{tmp:quadratic} hold,  and  $\bfSigma_\one$ and $\bfSigma_\two$ are of full rank. We have the following results:

\noindent 1).   \cref{aspt:ergodic} holds with $\rho=\frac12\lambda_-$ and $D_e=\frac{1}{2\lambda_-}(|\bfF_\one|^2+|\bfF_\two|^2).$\\
\noindent 2). \cref{aspt:control} holds with $v=1,m=1$ and
\[
D_c=\max\left\{ 1,  \frac{2\lambda_+\sigma^{-2}_{\two,-}}{1-\exp(-2\lambda_+)}, \frac{\sigma^2_{\two,+}}{2\lambda_-}, 2\lambda^2_+\sigma^{-2}_{\one,-},2\lambda^2_B \sigma^{-2}_{\one,-},\exp(2\lambda_+)\right\},
\]
where the constants are chosen such that  $|\bfB_{\two,1}(\uone)|\leq \lambda_B|\uone|$ and
\[
\sigma^{2}_{\one,-} I_{N_\one}\preceq \bfSigma_\one\bfSigma_\one^*,\quad \sigma^{2}_{\two,-} I_{N_\two}\preceq \bfSigma_\two\bfSigma_\two^*\preceq \sigma^{2}_{\two,+} I_{N_\two}.
\]

\end{prop}
The proof of \cref{prop:SF} is shown in \ref{Sec:Proof_Stochastic_Model}. The energy conservation property plays an essential role in verifying the system stability, and

\section{Proofs}\label{Sec:Proofs}

\subsection{Finite time MISE}
\begin{proof}[Proof of \cref{thm:MISE}]
Denote the one sample path density function:
\[
\hat{p}_i(\uone,\utwo):=K_H (\uone- \uone^i(t))p(\utwo|\uone^i(s\leq t)),
\]
such that the recovered PDF is given by $\hat{p}_t(x, y)=\frac1L \sum_{i=1}^L \hat{p}_i(x,y)$. Consider its average
\[
\bar{p}_t(\uone,\utwo)=\E K_H (\uone- \uone(t))p(\utwo|\uone^i(s\leq t))=\E \hat{p}_t(\uone,\utwo).
\]
The true density can be written as $p_t(\uone,\utwo)=\E \delta_{\uone^i(t)}(\uone)p(\utwo|\uone^i(s\leq t))$, since for any test function $f$, the following holds
\begin{align*}
\int p_t&(\uone,\utwo)f(\uone,\utwo)d\uone d\utwo=\E f(\uone^i(t), \utwo^i(t))\\
&=\E \E(f(\uone^i(t), \utwo^i(t))| \uone^i(s\leq t))=\E \int f(\uone,\utwo)\delta_{\uone^i(t)}(\uone)p(\utwo|\uone^i(s\leq t))d\uone d\utwo.
\end{align*}
This gives the following result
\begin{align*}
\bar{p}_t(\uone,\utwo)&= \E K_H(\uone- \uone(t))p(\utwo|\uone^i(s\leq t))\\
&=\E \int d\uone' K_H (\uone- \uone') \delta_{\uone^{i}(t)}(\uone') p(\utwo|\uone^i(s\leq t))\\
&= \int d\uone' K_H (\uone- \uone') p_t(\uone',\utwo)=:K_H* p_t (\uone,\utwo),
\end{align*}
where $*$ denotes the convolution. The Variance-Bias decomposition of the MISE can be made:
\begin{align}
\notag
&\E \int |\hat{p}_t(\uone,\utwo)-p_t(\uone, \utwo)|^2 d\uone d\utwo\\
\notag
&=  \int \E |\hat{p}_t(\uone,\utwo)-\bar{p}_t(\uone,\utwo)|^2d\uone d\utwo+\int |\bar{p}_t(\uone,\utwo)-p_t(\uone, \utwo)|^2 d\uone d\utwo\\
\label{eqn:MISEdecomp}
&= \int \text{var}\,\, \hat{p}_t(\uone, \utwo) d\uone d\utwo+ \int |\bar{p}_t(\uone,\utwo)-p_t(\uone,\utwo)|^2 d\uone d\utwo.
\end{align}
Since $\bar{p}_t=p_t*K_H$, so
\[
|\bar{p}_t(\uone,\utwo)-p_t(\uone,\utwo)|=\left|\int K_H(\uone-\uone')(p_t(\uone',\utwo)-p_t(\uone,\utwo)) d\uone'\right|.
\]
In \cref{lem:kernel}, a Taylor expansion on $(p_t(\uone',\utwo)-p_t(\uone,\utwo))$ leads to the following upper bound for the bias part:
\[
 \frac{1+\delta}{4} H^2 J\left( \sum_{i=1}^{N_\one} c_i^2\partial^2_{\bfu_{\one,i}^2} p_t(\uone,\utwo)\right)+\frac{1+\delta^{-1}}2 M^2 H^3\left(\sum_{i=1}^{N_\one}c_i^2\right)^3 J(M(\uone,\utwo)),\quad \forall \delta>0.
\]
Moreover, in light of the relation $\hat{p}_t(\uone, \utwo)=\frac1L \sum_{i=1}^L \hat{p}_i(\uone,\utwo)$ and the independence of the density samples $\hat{p}_i$, we have
\begin{align}
\notag
\int \text{var}\,\, \hat{p}_t(\uone, \utwo) d\uone d\utwo&=\frac1L \int \text{var} \,\,\hat{p}_i(\uone,\utwo) d\uone d\utwo\\
\notag
&\leq \frac1L \int \E |\hat{p}_i(\uone,\utwo)|^2 d\uone d\utwo=\frac1L  \E\int  |\hat{p}_i(\uone,\utwo)|^2 d\uone d\utwo.
\end{align}
Note that each $\hat{p}_i(x,y)$ is a Gaussian density with mean $(\uone^{i}(t), \barut(t))$ and a block diagonal covariance, where the blocks are given by $H \bfC$ and $\bfR_\two(t)$, respectively. In \cref{lem:L2Gaussian}, a straightforward computation of the $L^2$ norm of a Gaussian density shows that
\[
\int |\hat{p}_i(\uone,\utwo)|^2 d\uone d\utwo=\frac{1}{\sqrt{\prod_{i=1}^{N_\one} (\pi Hc_i^2) \text{det}(\pi \bfR_\two(t))} } .
\]
This leads to the bound of the MISE.
\end{proof}

\begin{proof}[Proof of  \cref{prop:marginal}]
Denote $\hat{p}_i(\utwo)=p(\utwo| \uone^{i}(s\leq t))$, then following the same proof as in \cref{thm:MISE}, we have $p_t(\utwo)=\E \hat{p}_i (\utwo)$ and $\hat{p}_t(\utwo)=\frac1L \sum_{i=1}^L \hat{p}_i(\utwo)$. Thus,
\begin{align*}
\int  |p_t(\utwo)-\hat{p}_t(\utwo)|^2 d\utwo=\int \text{var} \,\,&\hat{p}_t(\utwo) d\utwo=\frac1L \int \text{var} \,\,\hat{p}_i(\utwo) d\utwo\\
&\leq \frac1L \int \E |\hat{p}_i(\utwo)|^2 d\utwo= \frac1L \E \frac{1}{\sqrt{\text{det}(\pi \bfR_\two(t))}}.
\end{align*}
\end{proof}

\begin{proof}[Proof of  \cref{cor:marginal}]
The proof is identical to the one of  \cref{thm:MISE}, as long as one replaces the densities involving $\utwo$ to the version for $\utwo^P$. Therefore it is omitted here.
\end{proof}

\subsection{Long time result}

\begin{proof}[Proof of \cref{thm:ergodic}]
Part 1): The geometric ergodicity, i.e. the following $L^1$ convergence,
\[
\int |p_t(\uone,\utwo)-p_\infty(\uone,\utwo)| d\uone d\utwo\leq D_0 e^{-ct} \langle |\bfu|^2+1, p_0\rangle,
\]
 is a direct result that comes from the framework of \cite{meyn2012markov, mattingly2002ergodicity}. Its equivalence to the Poincar\'e type of inequality \eqref{tmp:Poincare} is a result by \cite{bakry2008rate}. We will try to verify the conditions needed in \cite{bakry2008rate}.

We claim that $V(\uone,\utwo)=|\uone|^2+|\utwo|^2+1$ is a Lyapunov function of Definition 1.1 in \cite{bakry2008rate}. Apply the generator $\mathcal{L}$ of the diffusion process
\begin{equation*}
\begin{split}
\mathcal{L} V&= 2\uone  \cdot (\bfA_0+\bfA_1\utwo) +2 \utwo\cdot (\bfa_0+\bfa_1\utwo)+\text{tr}(\bfSigma_\one \bfSigma_\one^*+\bfSigma_\two \bfSigma_\two^*)\\
&\leq -2\rho V+(2\rho+2D_e+\text{tr}(\bfSigma_\one \bfSigma_\one^*+\bfSigma_\two \bfSigma_\two^*))\leq -\rho V+b \unit_{\mathcal{U}},
\end{split}
\end{equation*}
where $b=2\rho+2D_e+\text{tr}(\bfSigma_\one \bfSigma_\one^*+\bfSigma_\two \bfSigma_\two^*)$, and $\mathcal{U}=\{V(\uone,\utwo)\leq b\}$.
The fact that $\mathcal{U}$, and actually any compact subset, is a petite set can be verified by the same proof of Lemma 3.4 in \cite{mattingly2002ergodicity}, since we assume $\bfSigma_\one$ and $\bfSigma_\two$ are full rank. The fact the stochastic process is irreducible can  also be verified using the same argument. More details  on these arguments  are provided in \cite{majda2016ergodicity} for more general conditions.

Therefore, applying theorem 1.2 of \cite{bakry2008rate}  leads to the $L^1$ convergence above. Theorem 2.1 also applies with $f(\uone,\utwo)=\frac{p_0}{p_\infty}(\uone,\utwo)$, which gives \eqref{tmp:Poincare}.
\newline

\noindent Part 2): We again decompose the MISE into  \eqref{eqn:MISEdecomp}.
\[
\text{MISE}=\int \text{var}\,\, \hat{p}_t(\uone, \utwo) d\uone d\utwo+ \int |\bar{p}_t(\uone,\utwo)-p_t(\uone,\utwo)|^2 d\uone d\utwo.
\]
Following the proof of  \cref{thm:MISE}, we have the variance part
\[
\int \text{var}\hat{p}_t(\uone,\utwo) d\uone d\utwo\leq  \E \frac{1}{L\sqrt{\prod_{i=1}^{N_\one} (\pi Hc_i^2) \text{det}(\pi \bfR_\two(t))} } .
\]
\cref{prop:control} leads to $\E \frac{1}{\sqrt{\text{det}(\bfR_\two(t))}}\leq D_1+D_2\int^t_{t-v} \E |\uone(r)|^{mN_\two}dr.$ To provide a bound for $\E |\uone(t)|^{mN_\two}$,  we verify that any fixed moment $|\bfu|^{2n}=(|\uone|^2+|\utwo|^2)^n$ is also dissipative. Applying the generator of the diffusion process yields
\begin{align*}
&\mathcal{L}|\bfu(t)|^{2n}=2n|\bfu|^{2(n-1)}( \uone\cdot (\bfA_0+\bfA_1\utwo) +\utwo\cdot(\bfa_0+\bfa_1\utwo))\\
&\qquad\qquad\qquad+n\text{tr}(\bfSigma^*(|\bfu|^{2(n-1)}I+2(n-1)|\bfu|^{2(n-2)}\bfu \bfu^*)\bfSigma)\\
&\leq -2n \rho|\bfu|^{2n}+2n D_e |\bfu|^{2(n-1)}+2n^2 \text{tr}(\bfSigma_\one \bfSigma_\one^*+\bfSigma_\two \bfSigma_\two^*))|\bfu|^{2(n-1)}\leq -n\rho |\bfu|^{2n}+D_{n,\bfSigma},
\end{align*}
where $\bfSigma=[\bfSigma_\one^*,\bfSigma_\two^*]^*$ and the constant $D_{n,\bfSigma}$ exists because of Young's inequality.

Apply Dynkin's formula for $e^{\rho nt} |\bfu(t)|^{2n}$, and combine it with the result above,  we have the following Gronswall's inequality
\begin{equation}
\label{tmp:highmonent}
\E |\bfu(t)|^{2n}\leq  e^{-\rho nt}\E |\bfu(0)|^{2n}+\frac{D_{n,\bfSigma}}{n\rho}.
\end{equation}
To continue, we let $n= mN_\two/2$ in \eqref{tmp:highmonent} and integrate it in time range $[t-v,t]$,
\[
\E \int^t_{t-v} |\uone(s)|^{mN_\two} ds\leq v\exp(-\tfrac12\rho mN_\two (t-v))\E |\bfu(0)|^{mN_\two}+\frac{2vD_{mN_\two/2,\bfSigma}}{mN_\two\rho}.
\]
Consequently, there exists a constant  $D_{m,N_\two,v}$ such that
\[
\int \text{var}\hat{p}_t(\uone,\utwo) d\uone d\utwo\leq \frac{D_{m,N_\two,v}}{L \pi^{\frac{N_\one+N_\two}{2}}H^{\frac{N_\one}{2}}\prod_{i=1}^{N_\one} c_i}\left(\exp(-\tfrac12\rho mN_\two t) \E |\bfu(0)|^{mN_\two}+D_{m,N_\two,v}\right).
\]
For the bias term, $\int |\bar{p}_t(\uone,\utwo)-p_t(\uone,\utwo)|^2 d\uone d\utwo$,  we use the Cauchy Schwartz
\[
(a+b+c)^2\leq \left(\frac{1}{1+\delta}+\frac{\delta}{2(1+\delta)}+\frac{\delta}{2(1+\delta)}\right)\left( (1+\delta) a^2+2(1+\delta^{-1})b^2+2(1+\delta^{-1})c^2\right),
\]
with
\[
a=|p_t(\uone,\utwo)-p_\infty(\uone,\utwo)|, ~ b=|\bar{p}_t(\uone,\utwo)-\bar{p}_\infty(\uone,\utwo)|,~ c=|p_\infty(\uone,\utwo)-\bar{p}_\infty(\uone,\utwo)|.
\]
Recall that $\bar{p}_\infty=K_H*p_\infty$. Using the same proof as in \cref{thm:MISE}, we have
\begin{multline*}
\int |p_\infty(\uone,\utwo)-\bar{p}_\infty(\uone,\utwo)|^2 d\uone d\utwo \\\leq \frac{1+\delta}{4} H^2 R\left( \sum_{i=1}^{N_\one} c_i^2\partial^2_{\bfu_{\one,i}^2} p_\infty(\uone,\utwo)\right)+\frac{1+\delta^{-1}}2  H^3\left(\sum_{i=1}^{N_\one}c_i^2\right)^3 .
\end{multline*}
Then apply \eqref{tmp:Poincare}, we have
\begin{align*}
\int |p_t(\uone,\utwo)&-p_\infty(\uone,\utwo)|^2d \uone d \utwo\\
&\leq \|p_\infty\|_\infty \int |p_t(\uone,\utwo)-p_\infty(\uone,\utwo)|^2\frac{1}{p_\infty(\uone,\utwo)}d \uone d \utwo\\
&\leq D_0 e^{-ct}\langle |\bfu|^2+1, p_0\rangle \left\|\tfrac{p_0}{p_\infty}-1\right\|_\infty^2 \|p_\infty\|_\infty.
\end{align*}
Next, recall that $\bar{p}_t(\uone,\utwo)=\int K_H(\uone') p_t(\uone-\uone',\utwo) d\uone'$. Therefore, by Cauchy Schwartz
\begin{align*}
|\bar{p}_t(\uone,\utwo)&-\bar{p}_\infty(\uone,\utwo)|^2 \\
&=\left(\int K_H(\uone') (p_t(\uone-\uone',\utwo)-p_\infty(\uone-\uone',\utwo)) d\uone'\right)^2 \\
&\leq \int K_H(\uone') d\uone' \int (p_t(\uone-\uone',\utwo)-p_\infty(\uone-\uone',\utwo))^2 K_H(\uone') d\uone'\\
&\leq \int (p_t(\uone-\uone',\utwo)-p_\infty(\uone-\uone',\utwo))^2 K_H(\uone') d\uone'.
\end{align*}
Consequently,
\begin{align*}
\int |\bar{p}_t(\uone,\utwo)&-\bar{p}_\infty(\uone,\utwo)|^2 d\uone d\utwo\\
&\leq \int (p_t(\uone-\uone',\utwo)-p_\infty(\uone-\uone',\utwo))^2 K_H(\uone') d\uone'd\uone d\utwo\\
&=\int \left(\int (p_t(\uone-\uone',\utwo)-p_\infty(\uone-\uone',\utwo))^2  d\uone d\utwo\right)K_H(\uone-\uone')d\uone'\\
&=\int |p_t(\uone,\utwo)-p_\infty(\uone,\utwo)|^2d \uone d \utwo.
\end{align*}
Combining the results finishes the proof.
\end{proof}

\section{Numerical examples}\label{Sec:Numerics}
Below, numerical examples are used to support the theoretical results in \cref{Sec:Theorems}. The test model considered here is the following \emph{triad model} \cite{majda2012physics},
\begin{subequations}\label{TriadModel}
\begin{align}
  \frac{du_1}{dt} &= A_1u_2u_3,\label{TriadModel_u1}\\
  \frac{du_2}{dt} &= A_2u_3u_1 - d_2u_2 + \sigma_2\dot{W}_2,\label{TriadModel_u2}\\
  \frac{du_3}{dt} &= A_3u_1u_2 - d_3u_3 + \sigma_3\dot{W}_3,\label{TriadModel_u3}
\end{align}
\end{subequations}
where $A_1+A_2+A_3=0$ represents the energy-conserving nonlinear interactions and $d_2>0, d_3>0$ are the damping terms. Note that there is no damping and dissipation in \eqref{TriadModel_u1} but \eqref{TriadModel} is a hypoelliptic diffusion \cite{majda2016ergodicity, mattingly2002ergodicity}. Linear stability is satisfied for $u_2, u_3$ while there is only neutral stability of $u_1$. Define $E_2 = \sigma_2^2/(2d_2)$ and $E_3 = \sigma_3^2/(2d_3)$. It is straightforward to show that the triad system \eqref{TriadModel} has a Gaussian invariant measure \cite{majda2002priori, majda2012physics}
\begin{equation*}
  p_{eq}(u) = C\exp\left(-\frac{1}{2}\left(\frac{u_1^2}{E_1} +\frac{u_2^2}{E_2} +\frac{u_3^2}{E_3} \right)\right),
\end{equation*}
provided that the following condition is satisfied
\begin{equation}\label{Triad_Stable_Condition}
  E_1 = -A_1 E_2 E_3 (A_2 E_3 + A_3 E_2)^{-1} > 0.
\end{equation}
If the condition in \eqref{Triad_Stable_Condition} is violated, namely $E_1<0$, then the variance in $u_1$ direction will increase unboundedly and there is no invariant measure for the triad system \eqref{TriadModel}.

Below, two dynamical regimes of the triad model \eqref{TriadModel} are studied, where the corresponding parameters are listed in the \cref{Table:Regimes}. Particularly, the triad system \eqref{TriadModel} in Regime I has a Gaussian invariant measure while there is no invariant measure in Regime II due to the fact that $E_1<0$. See Figure \ref{fig:Regimes} for the time evolution of the three marginal variances and one realization of each variable and \cite{majda2016introduction} for dynamical introduction about such triad models.

\begin{table}[tbhp]
\caption{Parameters of two dynamical regimes of the triad model \eqref{TriadModel}}
\label{Table:Regimes}
\centering
\begin{tabular}{|l|c|c|c|c|c|c|c||c||c|c|c|c|}
  \hline
            & $A_1$  & $A_2$ & $A_3$ & $d_2$ & $d_3$ & $\sigma_2$ & $\sigma_3$  &                   & $E_2$ & $E_3$ & $E_1$     & $\mbox{Var}(u_1)$\\
  Regime I  & $-2.5$ & $1$   & $1.5$ & $1$   & $0.5$ & $1$        & $1$         & $\Longrightarrow$ & $0.5$  & $1$  & $5/11$  & Bounded\\
  Regime II & $-0.5$ & $-1$  & $1.5$ & $1$   & $0.5$ & $1$        & $1$         &                   & $0.5$  & $1$  & $-5/3$  & Unbounded\\
  \hline
\end{tabular}
\end{table}

Denote $\mathbf{u}_\mathbf{I} = (u_2,u_3)^T$ and $\mathbf{u}_\two = u_1$. The triad system \eqref{TriadModel} belongs to the conditional Gaussian family \eqref{Conditional_Gaussian_System}. Notably, the noise coefficient in $\mathbf{u}_\two$ is $\boldsymbol\Sigma_\two=0$, which implies the system has no controllability.  The initial values in the tests below are all given at origin. Here only the hybrid method \eqref{Joint0} is tested and the number of samples is always $L=500$.

\cref{fig:RegimeI_t1} shows the recovered PDF at $t=1$ in Regime I of the triad model. Despite an accurate estimation of the joint PDF of the observed variables $p(u_2,u_3)$ as shown in Panel (e), the recovered PDF of the unobserved variable $u_1$ in Panel (f) has quite a few noisy fluctuations and the recovered joint PDFs $p(u_1,u_2)$ and $p(u_3,u_1)$ in Panel (d) and (f) are non-smooth in $u_1$ direction as well. Such  pathological behavior results from the loss of controllability of the system, which is consistent with the theoretical discussion in \cref{sec:lowerbound}. In fact, the term $\mathbf{a}_1$ in \eqref{Conditional_Gaussian_System} associated with the triad system \eqref{TriadModel} is zero. Therefore, according to \eqref{CG_Result}, $\boldsymbol\Sigma_\two=0$ implies the posterior variance $\mathbf{R}_\two=\mathbf{0}$ and the posterior mean $\mathbf{\bar{u}}_\two$ simply follows the sampled trajectory of $\mathbf{u}_\two$. In other words, the posterior states from the algorithm are exactly the Monte Carlo samples, as is validated in Panel (h). The same performance is found in Regime II and thus we omit the figure here.

In order to make the triad system have controllability, a small noise is added to \eqref{TriadModel_u1} and the resulting \emph{modified triad system} is given as follows,
\begin{subequations}\label{TriadModel_New}
\begin{align}
  \frac{du_1}{dt} &= A_1u_2u_3 + \epsilon\dot{W}_1,\label{TriadModel_u1_New}\\
  \frac{du_2}{dt} &= A_2u_3u_1 - d_2u_2 + \sigma_2\dot{W}_2,\label{TriadModel_u2_New}\\
  \frac{du_3}{dt} &= A_3u_1u_2 - d_3u_3 + \sigma_3\dot{W}_3,\label{TriadModel_u3_New}
\end{align}
\end{subequations}
where $\epsilon$ is the noise coefficient of $u_1$ with $\boldsymbol\Sigma_\two=\epsilon$ in \eqref{Conditional_Gaussian_System}. Below we set $\epsilon=0.1\ll \sigma_2=\sigma_3=1$. The other parameters in \eqref{TriadModel_New} remain the same as those in \cref{Table:Regimes}.

This extra noise implies the triad system is controllable, which significantly improves the accuracy of the recovered PDFs. See \cref{fig:RegimeI_t1_Noise01} for the results in Regime I at $t=1$. In particular, Panel (h) of \cref{fig:RegimeI_t1_Noise01} shows that the posterior means are quite different from the Monte Carlo samples and the posterior variances are no longer zero. It is also shown in \cref{fig:RegimeI_t20_Noise01} that the recovered PDFs at a long time $t=20$ (i.e., statistically steady state) are very close to the truth with this extra small noise.

Similarly, \cref{fig:RegimeII_t1_Noise01} shows the recovered PDFs of Regime II with $\epsilon=0.1$ at $t=1$, the error in which compared with the truth is negligible. Notably, although the amplitude of $u_1$ has an unbounded growth in this regime due to the fact that $E_1<0$, the recovered PDFs with $\epsilon=0.1$ at $t=20$ as illustrated in \cref{fig:RegimeII_t20_Noise01} remain quite accurate. Next, the performance of the hybrid algorithm at a very long time in this regime is studied. \cref{fig:RegimeII_t400_Noise01} shows the recovered PDFs at $t=400$. Similar to \cref{fig:RegimeI_t1}, the noisy fluctuations are found in the recovered PDF of $u_1$. In fact, direct calculations show  that the posterior variance $\mathbf{R}_\two$ in \eqref{CG_Result} is bounded from above since the unbounded signal $u_1$ does not enter into the evolution of $\mathbf{R}_\two$, which is also validated by the numerical simulation in Panel (h). Since the variance of $u_1$ increases with time, the percentage of the portion covered by each conditional Gaussian distribution decreases in time, which reduces the skill in the recovered PDFs by the conditional Gaussian mixtures.
In \cref{fig:RegimeII_damping_t50}, we show that by further imposing a damping in the dynamics of $u_1$ of the modified triad model \eqref{TriadModel_New} in Regime II, the model then satisfies all the conditions in \cref{prop:SF} and the resulting  model has an invariant measure. In such a scenario, the hybrid algorithm is skillful in both short and long time as is affirmed by \cref{prop:SF}. 

It is also worthwhile pointing out that all the test models in \cite{chen2017efficient}, including the noisy version of Lorenz 63 model \cite{lorenz1963deterministic}, the stochastic climate model \cite{majda2008applied, majda2005information}, the nonlinear triad model mimicking structural features of low-frequency variability of GCMs with non-Gaussian features \cite{majda2009normal} and the modified conceptual dynamical model for turbulence \cite{majda2014conceptual}, all satisfy the conditions in \cref{prop:SF}. Therefore, the hybrid algorithm \eqref{Joint0} is able to solve the PDFs of those models with high accuracy with only a small number of samples.

\begin{figure}
  \centering\label{fig:Regimes}
  \hspace*{-2cm}\includegraphics[width=17.5cm]{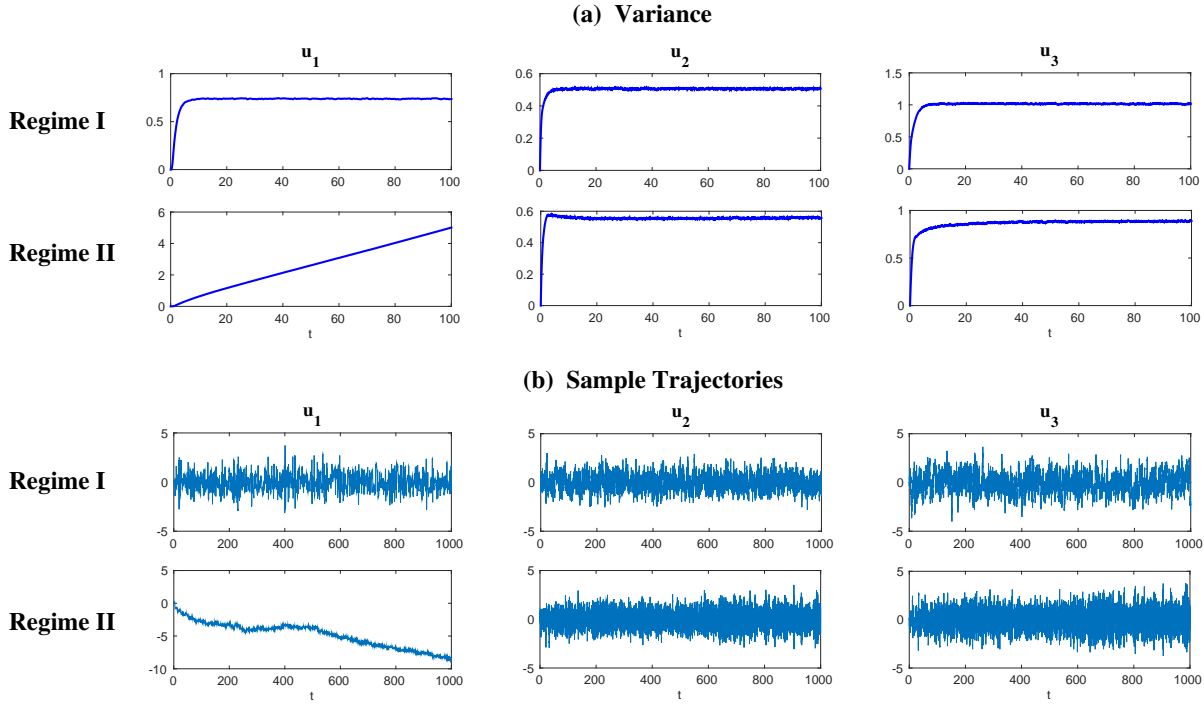}
  \caption{Triad model \eqref{TriadModel}. (a) Marginal variance as a function of time ($t\in[0,100]$) in the two dynamical regimes with parameters in \cref{Table:Regimes}. (b) Sample trajectories up to $t=1000$ of the two dynamical regimes. Note the unbounded growth of the amplitude of $u_1$ in Regime II.}
\end{figure}

\begin{figure}
  \centering\label{fig:RegimeI_t1}
  \hspace*{-2cm}\includegraphics[width=17.5cm]{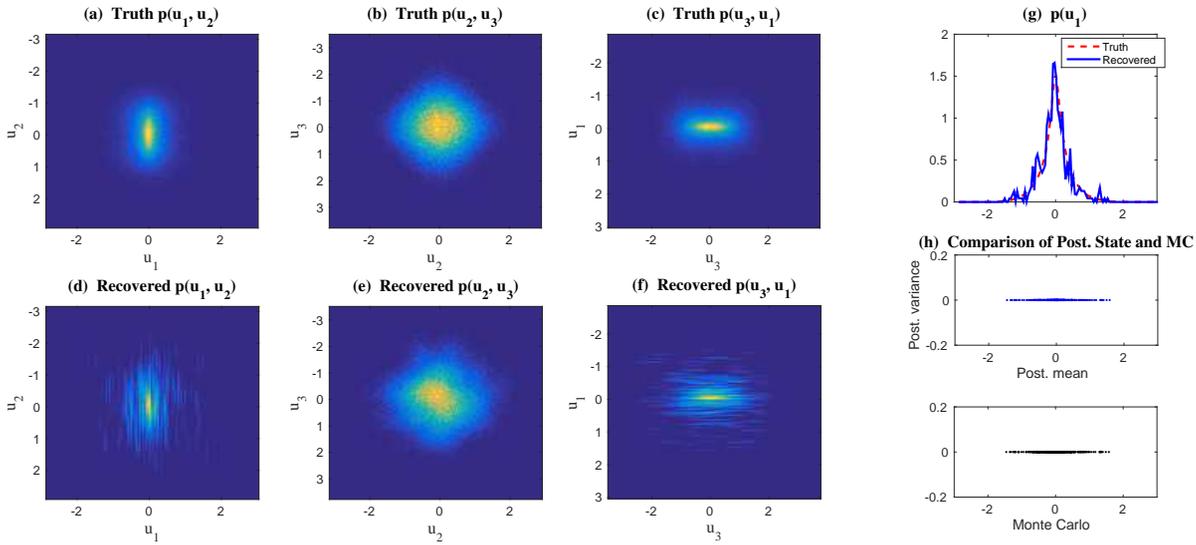}
  \caption{Triad model \eqref{TriadModel}, Regime I at $t=1$. (a)-(c) True 2D PDF. (d)-(f) Recovered PDF. (g) True and recovered 1D PDF $p(u_1)$. (h) Top: Posterior mean (x-axis) and posterior variance (y-axis). Bottom: Monte Carlo samples.  The total number of samples is $L=500$.}
\end{figure}

\begin{figure}
  \centering\label{fig:RegimeI_t1_Noise01}
  \hspace*{-2cm}\includegraphics[width=17.5cm]{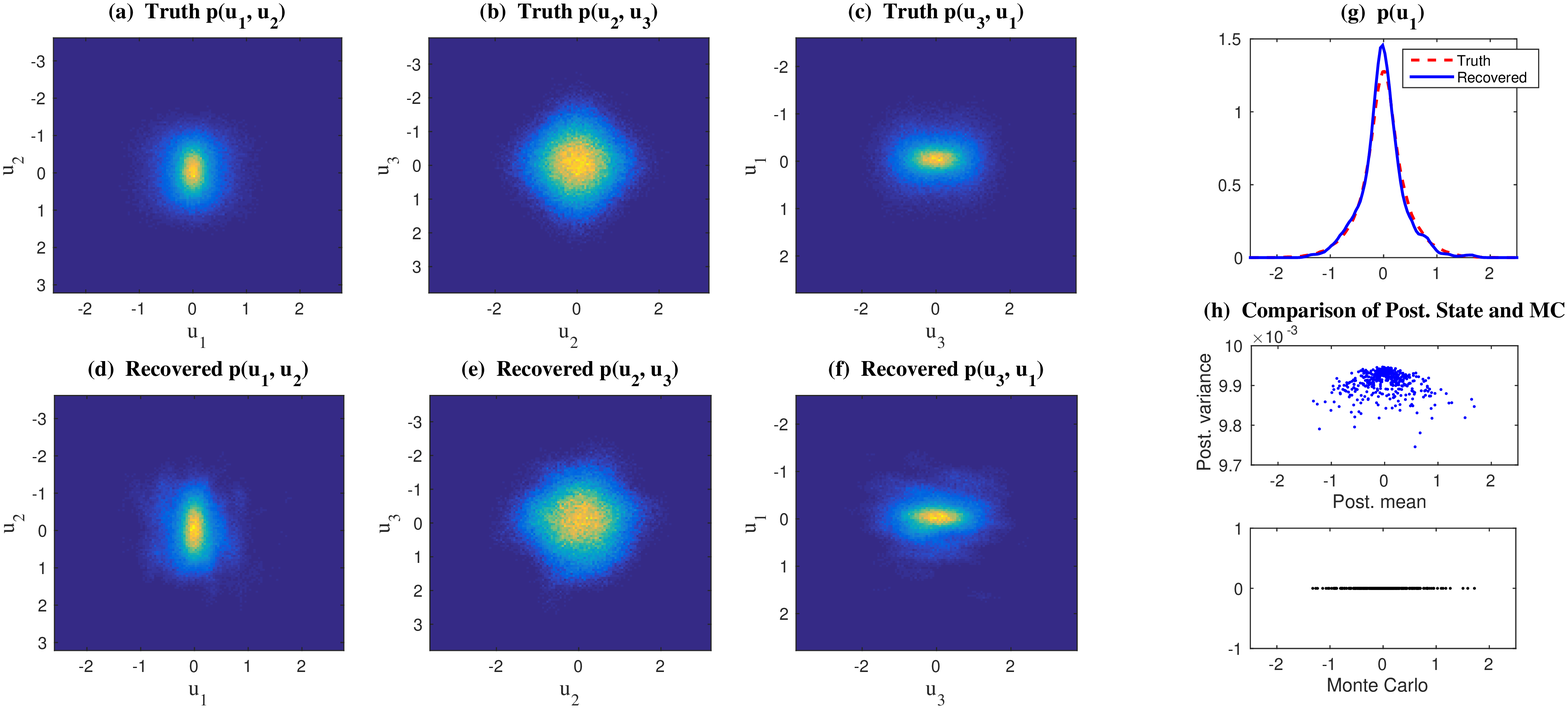}
  \caption{Modified triad model \eqref{TriadModel_New}, Regime I at $t=1$. Same captions as in \cref{fig:RegimeI_t1}.}
\end{figure}

\begin{figure}
  \centering\label{fig:RegimeII_t1_Noise01}
  \hspace*{-2cm}\includegraphics[width=17.5cm]{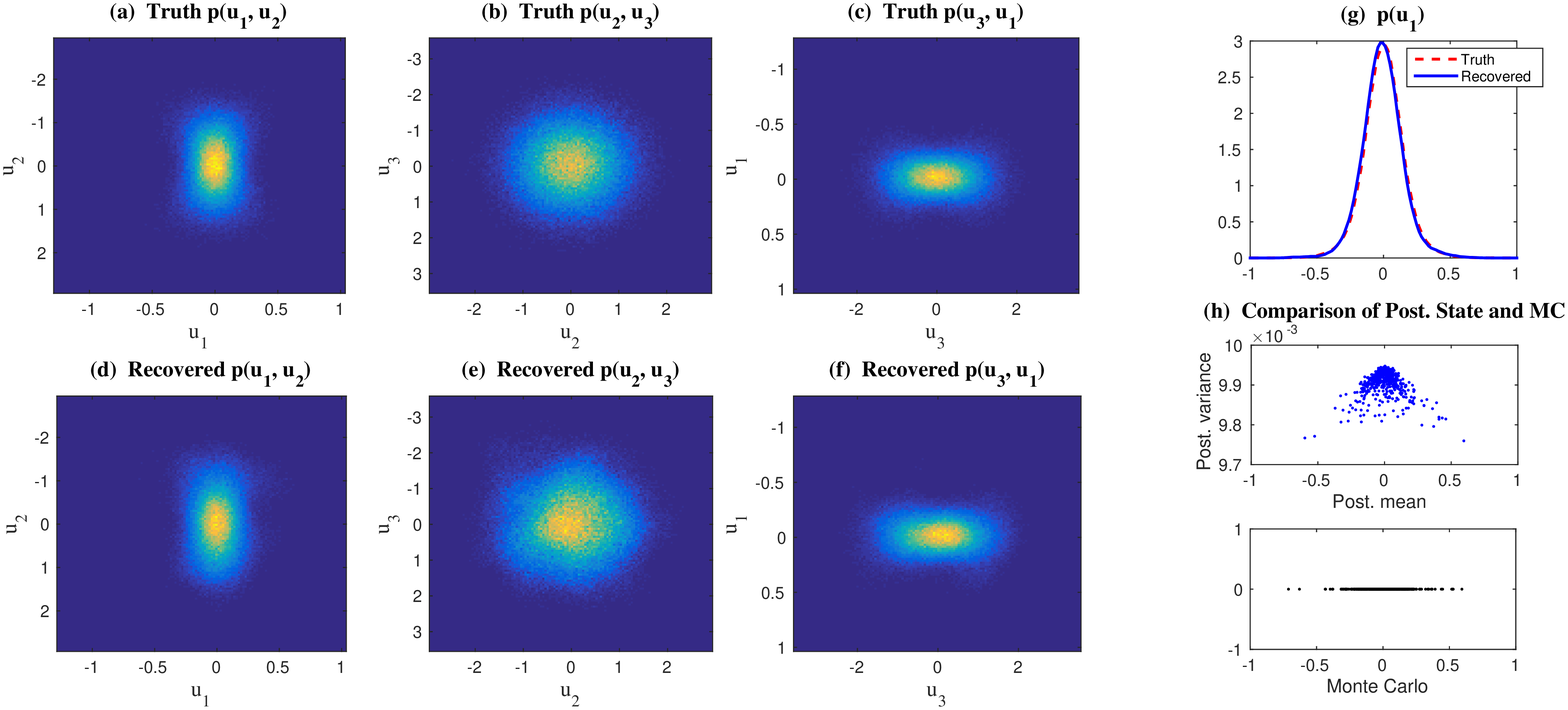}
  \caption{Modified triad model \eqref{TriadModel_New}, Regime II at $t=1$. Same captions as in \cref{fig:RegimeI_t1}.}
\end{figure}

\begin{figure}
  \centering\label{fig:RegimeII_t400_Noise01}
  \hspace*{-2cm}\includegraphics[width=17.5cm]{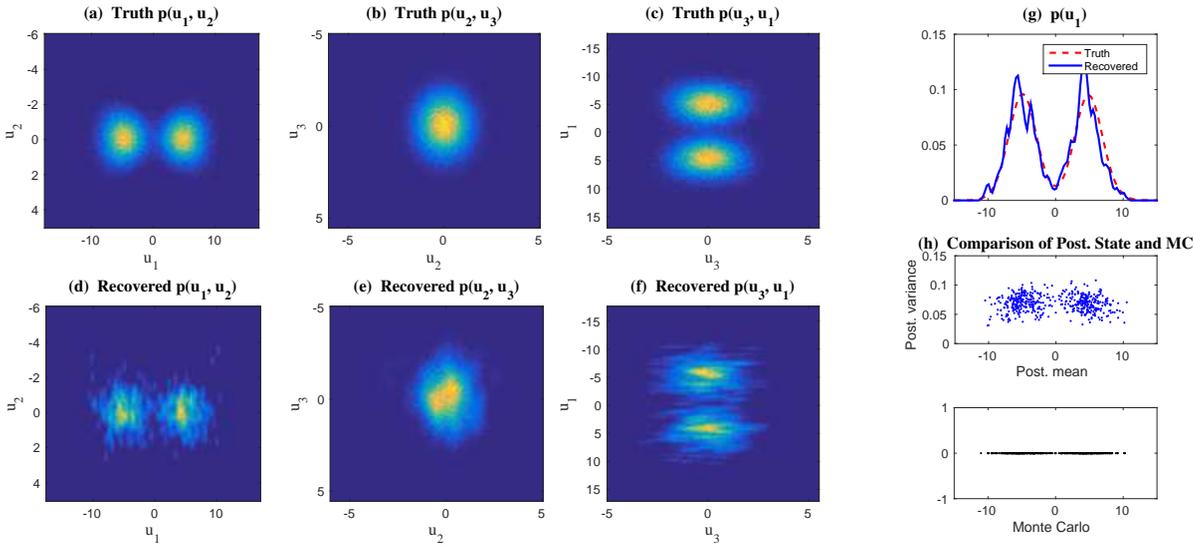}
  \caption{Modified triad model \eqref{TriadModel_New}, Regime II at $t=400$. Same captions as in \cref{fig:RegimeI_t1}.}
\end{figure}

\section{Discussion and Conclusions}\label{Sec:Conclusion}

This article presents a rigorous analysis for the efficient statistically accurate algorithms developed in \cite{chen2017efficient}, which succeed in  solving both the transient and the equilibrium solutions of Fokker-Planck equations associated with high-dimensional nonlinear turbulent dynamical systems with conditional Gaussian structures. Despite the conditional Gaussianity, these nonlinear systems capture many strong non-Gaussian features such as intermittency and fat-tailed PDFs. The algorithms involve a hybrid strategy that requires only a small number of samples $L$ to capture both the transient and the equilibrium non-Gaussian PDFs with high accuracy.

\cref{thm:MISE} shows that the MISE in the recovered high-dimensional PDFs associated with the unresolved variables $\utwo$ is bounded by $\E (\text{det}(\bfR_\two)^{-1/2})$, where $\bfR_\two$ is completely determined by the underlying dynamical systems and it has no dependence on the sample size $L$. This is fundamentally different from the direct application of the kernel methods to recover the PDF of $\utwo$, in which the bandwidth of the kernel $H$ is scaled as a reciprocal of  $L$ to a certain power and the resulting MISE is proportional to $L^{-1/N_\two}$.
This implies the curse of dimensionality in the kernel density estimation and other smoothed Monte Carlo methods due to the fact that $L$ has to increase exponentially as $N_\two$ in order to guarantee the accuracy in the solution. As is shown in \cref{thm:MISE}, many fewer samples are needed in the efficient statistically accurate algorithms in order to reach the same accuracy as using the smoothed Monte Carlo methods, especially with a large $N_\mathbf{II}$.
\cref{thm:ergodic} affirms the long term persistence of the efficient statistically accurate algorithms in a rigorous way under the assumption that the joint process $(\mathbf{u}_\mathbf{I},\mathbf{u}_\two)$ is controllable and stochastically stable. It also provides a lower bound of $\mathbf{R}_\two$ using the controllability condition. The validations of the controllability and other theoretical conditions in the algorithms are demonstrated in the numerical simulations in \cref{Sec:Numerics}.
Furthermore,  \cref{prop:SF} illustrates that the turbulent dynamical systems with quadratic energy conserving nonlinear interactions \cite{majda2016introduction} automatically satisfy all the conditions for the long time persistence. This justifies the skillful performance of the efficient statistically accurate algorithms in the numerical tests reported in \cite{chen2017efficient} and provides important guidelines for future applications.

\section*{Appendix}
\appendix
This appendix contains the following information. Section \ref{Sec:TwoLemmas} states and proves two lemmas that are needed to prove \cref{thm:MISE}. Section \ref{sec:proofcontrol} shows the proof of \cref{prop:control} regarding the controllability and observability and Section \ref{sec:contract} includes the discussions of the contraction of the Riccati flow. The controllability and long time behavior of the conditional Gaussian turbulent dynamical systems with energy-conserving quadratic nonlinearity in \cref{prop:SF} are demonstrated in Section \ref{Sec:Proof_Stochastic_Model}. Finally, extra numerical examples of the triad model and modified version of the triad model are shown in \ref{Sec:ModifyModel}.

\section{Two lemmas for \cref{thm:MISE}}\label{Sec:TwoLemmas}

\begin{lem}
\label{lem:L2Gaussian}
Let $p(\utwo)$ be the PDF of $\mathcal{N}(\bfa, \bfSigma) $, its $L^2$ norm is
\[
\int p^2 (\utwo)d \utwo =\frac{1}{\sqrt{\text{det} (\pi \bfSigma)}}.
\]
\end{lem}

\begin{proof}
Note the Gaussian density has form
\[
p(\utwo)=\frac{1}{\sqrt{\text{det}(2\pi \bf\Sigma)}}\exp\left(-\frac12 (\utwo-\bfa)\cdot \bfSigma^{-1}(\utwo-\bfa)\right).
\]
Its square can be decomposed as
\[
p^2(\utwo)=\frac{1}{\sqrt{\text{det}(\pi \bfSigma )}}\cdot \frac{1}{\sqrt{\text{det}(4\pi \bfSigma)}}\exp\left(- (\utwo-\bfa)\cdot \bfSigma^{-1}(\utwo-\bfa)\right).
\]
The second part is the Gaussian density $\mathcal{N}(\bfa, 2\bfSigma)$, so its integral is one. This concludes our proof.
\end{proof}

\begin{lem}
\label{lem:kernel}
Suppose $p_t(\uone,\utwo)$ has bounded third derivative as \eqref{eqn:thirder}. Consider filtering it with kernel $K_H$ at the $\uone$ components, define
\[
\bar{p}_t(\uone,\utwo)=\int K_H(\uone-\uone')p_t(\uone',\utwo) d\uone' .
\]
Then $\int |\bar{p}_t(\uone,\utwo)-p_t(\uone,\utwo)|^2d\uone d\utwo$ is bounded by
\[
 \frac{1+\delta}{4} H^2 J\left( \sum_{i=1}^{N_\one} c_i^2\partial^2_{\bfu_{\one,i}^2} p_t(\uone,\utwo)\right)+\frac{1+\delta^{-1}}2 M^2 H^3\left(\sum_{i=1}^{N_\one}c_i^2\right)^3 J(M(\uone,\utwo)),\quad \forall \delta>0.
\]
\end{lem}

\begin{proof}
Apply Taylor's expansion to $p_t(\uone,\utwo)$,  use \eqref{eqn:thirder} with $\bfv=\frac{\uone'-\uone}{|\uone'-\uone|}$,
\begin{align*}
\left|p_t(\uone',\utwo)-p_t(\uone,\utwo)-|\uone'-\uone| \frac{d}{d s} p_t(\uone+s\bfv, \utwo)-  \frac{|\uone'-\uone|^2}{2}\frac{d^2}{d s^2} p_t(\uone+s\bfv, \utwo)\right| \\
\leq \frac{1}{6} |\uone'-\uone|^3 \int^1_0 \left|\frac{d^3}{d s^3} p_t(\uone+s\bfv, \utwo)\right|ds  \leq\frac{1}{6} |\uone'-\uone|^3 \int^1_0 M(\uone+s\bfv,\utwo) ds .
\end{align*}
Denote the gradient and Hessian with respect to $\uone$ as $\nabla_\one$ and $\nabla^2_\one$. Note that
\[
|\uone'-\uone| \frac{d}{d s} p_t(\uone+s\bfv, \utwo)=\nabla_{\one} p_t(\uone,\utwo)\cdot (\uone'-\uone),
\]
and
\[
|\uone'-\uone|^2\frac{d^2}{d s^2} p_t(\uone+s\bfv, \utwo)=(\uone'-\uone)\cdot[\nabla^2_{\one} p_t(\uone,\utwo)](\uone'-\uone).
\]
Therefore,
\begin{align*}
\bigg|p_t(\uone',\utwo)-p_t(\uone,\utwo)- \nabla_{\one}& p_t\cdot (\uone'-\uone) -  \frac{1}{2}(\uone'-\uone)\cdot\nabla^2_{\one} p_t(\uone'-\uone)\bigg|\\
&\leq \frac{1}{6} |\uone'-\uone|^3 \int^1_0 M(\uone+s\bfv,\utwo) ds .
\end{align*}
To continue, we write
\begin{align*}
|\bar{p}_t(\uone,\utwo)-p_t(\uone,\utwo)|&=\left|\int K_H(\uone-\uone')(p_t(\uone',\utwo)-p_t(\uone,\utwo)) d\uone'\right|\\
&\leq \left|\int K_H(\uone-\uone') \nabla_{\one} p_t(\uone,\utwo)\cdot(\uone'-\uone)  d\uone'\right|\\
&\quad+\frac12\left|\int K_H(\uone-\uone') (\uone'-\uone)\cdot\nabla^2_{\one} p_t(\uone,\utwo)](\uone'-\uone)d\uone'\right|\\
&\quad+\frac16\int K_H(\uone-\uone')|\uone-\uone'|^3 M(\uone+s\bfv,\utwo) d\uone' ds,
\end{align*}
where the range of $s$ is $[0,1]$.

Note that by the symmetry of $K_H(\uone)$ in $\uone$,
\[
\int K_H(\uone'-\uone) \nabla_{\one} p_t(\uone,\utwo)\cdot( \uone'-\uone) d\uone'=0.
\]
Note that $K_H(\uone'-\uone)$ is the density of $\mathcal{N}(\uone,H \bfC)$, where $\bfC$ is diagonal with diagonal entries $c^2_i$. If we let $Z$ be a random variable following this distribution
\begin{align*}
&\int K_H(\uone-\uone') (\uone'-\uone)\cdot [\nabla^2_{\one} p_t(\uone,\utwo)](\uone'-\uone)d\uone'\\
&= \E Z\cdot [\nabla^2_{\one} p_t]Z=\E \text{tr}([\nabla^2_{\one} p_t]ZZ^*)=H \text{tr}(\nabla^2_{\one} p_t\bfC)=
H \sum_{i=1}^{N_\one} c_i^2\partial^2_{\bfu_{\one,i}^2} p_t(\uone,\utwo).
\end{align*}
Therefore
\[
|\bar{p}_t(\uone,\utwo)-p_t(\uone,\utwo)|\leq   \frac12 H \left|\sum_{i=1}^{N_\one} c_i^2\partial^2_{\bfu_{\one,i}^2} p_t\right|+\frac16\int K_H(\uone-\uone')|\uone-\uone'|^3 M(\uone+s\bfv,\utwo) d\uone' ds.
\]
Using Young's inequality $(a+b)^2\leq (1+\delta)a^2+(1+\delta^{-1}) b^2$,  $|\bar{p}_t(\uone,\utwo)-p_t(\uone,\utwo)|^2$ is bounded by
\[
 \frac{1+\delta}{4} H^2\left( \sum_{i=1}^{N_\one} c_i^2\partial^2_{x_i^2} p_t(\uone,\utwo)\right)^2+\frac{1+\delta^{-1}}{36}\left(\int K_H(\uone-\uone')|\uone-\uone'|^3 M(\uone+s\bfv,\utwo) d\uone' ds\right)^2.
\]
Finally, by Cauchy Schwartz,
\begin{align*}
&\left(\int K_H(\uone-\uone')|\uone-\uone'|^3 M(\uone+s\bfv,\utwo) d\uone' ds\right)^2\\
&\leq  \left(\int K_H(\uone-\uone')|\uone-\uone'|^6 d\uone' ds\right)\left(\int K_H(\uone-\uone')M^2(\uone+s\bfv,\utwo) d\uone' ds\right).
\end{align*}
Note again $K_H$ is the density of $Z\sim \mathcal{N}(0, H\bfC)$,
\[
\int K_H(\uone-\uone') |\uone'-\uone|^6 d\uone'= \E |Z|^6=15H^3\left(\sum_{i=1}^{N_\one}c_i^2\right)^3.
\]
Therefore
\begin{align*}
&\int d\uone d\utwo \left(\int K_H(\uone-\uone')|\uone-\uone'|^3 M(\uone+s\bfv,\utwo) d\uone' ds\right)^2\\
&\leq  15H^3\int  d\uone d\utwo \int K_H(\uone-\uone')M^2(\uone+s\bfv,\utwo) d\uone' ds\\
&=15H^3 \int K_H(\uone-\uone') d\uone' ds \int d\uone d\utwo M^2(\uone+s\bfv,\utwo)\\
&=15H^3 J(M(\uone,\utwo))\int K_H(\uone-\uone') d\uone' ds=15H^3 \left(\sum_{i=1}^{N_\one}c_i^2\right)^3J(M(\uone,\utwo)).
\end{align*}
Combine these estimates with $\frac{15}{36}\leq \frac12$, we have our claimed bound.
\end{proof}

\section{Controllability and observability}
\label{sec:proofcontrol}

\begin{proof}[Proof of \cref{prop:control}]
Note that Riccati equation of $\bfR_\two(t)$ is given by
\[
\frac{d}{dt}\bfR_\two(t)=\bfa_1(t) \bfR_\two(t)+\bfR_\two(t) \bfa_1^*(t)+\bfSigma_\two\bfSigma_\two^*-\bfR_\two(t) \bfA_1^*(t) (\bfSigma_\one \bfSigma_\one^*)^{-1} \bfA_1(t) \bfR_\two(t).
\]
In \cite{bishop2016stability},  the matrix flow generated by this equation was studied. (In \cite{bishop2016stability}, it was termed as $\phi(t)(Q)$). Define the controllability and observability Gramian \cite{bishop2016stability}:
\[
\calC_{s,t}=\int^t_s \calE_{r,t} \bfSigma_\two\bfSigma_\two^*\calE_{r,t}^* dr.
\]
\[
\calO_{s,t}=\int^t_s(\calE_{r,t}^*)^{-1} \bfA^*_1(r,\uone(r)) [\bfSigma_\one\bfSigma^*_\one]^{-1}\bfA_1(r,\uone(r)) \calE_{r,t}^{-1}.
\]
Define also the following processes:
\[
\calD_{s,t}=\calC_{s,t}^{-1}\left[\int^t_s \calE_{r,t} \calC_{s,r} \bfA^*_1(r,\uone(r)) [\bfSigma_\one\bfSigma^*_\one]^{-1}\bfA_1(r,\uone(r)) \calC_{s,r} \calE^*_{r,t} dr\right] \calC_{s,t}^{-1}.
\]
\[
\calF_{s,t}=\calO_{s,t}^{-1}\left[\int^t_s (\calE^*_{r,t})^{-1} \calO_{s,r} \bfSigma_\two \bfSigma_\two^* \calO_{s,r} \calE^{-1}_{r,t} dr\right] \calO_{s,t}^{-1}.
\]
Theorem 4.4 in \cite{bishop2016stability} has shown that
\begin{equation}
\label{tmp:Rtbounds}
[\calD_{s,t}+\calC_{s,t}^{-1}]^{-1}\preceq \bfR_\two(t)\preceq\calO^{-1}_{s,t}+\calF_{s,t},
\end{equation}
by applying the comparison principal between the Riccati equation of $\bfR_\two(t)$ and the bounds in \eqref{tmp:Rtbounds}.

For the purpose of this proposition, we focus only on the left hand side of \eqref{tmp:Rtbounds}. First we find bounds for the controllability Gramian. Note that
\[
\calC_{s,t}\preceq \sigma^2_{\two, +} \int^t_s \calE_{r,t}\calE_{r,t}^* dr\preceq D_c (t-s)\sigma^2_{\two, +} I_{N_\two}.
\]
\[
\calC_{s,t}\succeq \sigma^2_{\two, -} \int^t_s \calE_{r,t}\calE_{r,t}^* dr\succeq D_c^{-1}\sigma^2_{\two, -} (t-s) I_{N_\two}.
\]
Then for $t-v\leq s\leq t$, apply  \cref{aspt:control} and the bounds above, we have
\begin{align*}
\calD_{s,t}&\preceq D_c  \calC_{s,t}^{-1}\left[\int^t_s (|\uone(r)|^{2m}+1) \calE_{r,t} \calC_{s,r}  \calC_{s,r} \calE^*_{r,t} dr \right] \calC_{s,t}^{-1}\\
&\preceq v^2 \sigma^2_{\two, +}D^3_c   \calC_{s,t}^{-1}\left[\int^t_s(|\uone(r)|^{2m}+1)\calE_{r,t}  \calE^*_{r,t} dr\right] \calC_{s,t}^{-1}\\
&\preceq v^2 \sigma^2_{\two, +}D^4_c   \left(\int^t_s (|\uone(r)|^{2m}+1)dr\right)\calC_{s,t}^{-2}\\
&\preceq v^2\sigma^2_{\two, +} \sigma^{-2}_{\two, -}D^6_c\left(t-s+\int^t_{s} |\uone(r)|^{2m}dr\right)I_{N_\two}.
\end{align*}
Consequentially, by taking $s=t-v$, we have $\bfR_\two(t)^{-1}\preceq h_{t,v}(\uone) I_{N_\two}$, where
\[
h_{t,v}(\uone):=v^2 \sigma^2_{\two, +} \sigma^{-2}_{\two, -}D^6_c\left(v+\int^t_{t-v} |\uone(r)|^{2m}dr\right) +v^{-1}D_c\sigma^{-2}_{\two, -}.
\]
In particular, this leads to $\sqrt{\text{det}\bfR_\two(t)}^{-1}\leq \sqrt{\det(h_{t,v}(\uone) I_{N_\two})}=h^{N_\two/2}_{t,v}(\uone)$. Note that for  for positive $a,b$ and $m\geq 1$,   $(a+b)^m\leq 2^m(a^m+b^m)$, we have
\begin{align*}
h^{\frac{N_\two}{2}}_{t,v}(\uone)\leq &2^{\frac{N_\two}{2}}\left(v^3\sigma^2_{\two, +} \sigma^{-2}_{\two, -}D^6_c+v^{-1}D_c\sigma^{-2}_{\two, -}\right)^\frac{N_\two}{2}\\
&+2^{\frac{N_\two}{2}}\left(v^2\sigma^2_{\two, +} \sigma^{-2}_{\two, -}D^6_c\right)^\frac{N_\two}{2}\left(\int^t_{t-v} |\uone(r)|^{2m}dr\right)^\frac{N_\two}{2}.
\end{align*}
Then applying Holder's inequality yields
\[
\left(\int^t_{t-v} |\uone(r)|^{2m}dr\right)^\frac{N_\two}{2}\leq v^{\frac{N_\two-2}{2}} \int^t_{t-v} |\uone(r)|^{mN_\two}dr.
\]
Combining the estimates above finishes the proof of \cref{prop:control}.
\end{proof}

In symmetry to \cref{prop:control}, we can also find an upper bound for $\bfR_\two(t)$. For that purpose,  let $\sigma_{A,-}^2(t)\geq 0$ be a stochastic process that satisfies:
\[
\sigma_{A,-}^2(t)I_{N_\two}\preceq\bfA^*_1(t,\uone(t)) [\bfSigma_\one\bfSigma^*_\one]^{-1}\bfA_1(t,\uone(t)).
\]
\begin{prop}
\label{prop:obs}
Under \cref{aspt:control}, for $t\geq v$, we have $\|\bfR_\two(t)\|\preceq g_{v,t}(\uone)$, where
\begin{multline*}
g_{v,t}(\uone):=D^2_c\left(v+\int^t_{t-v} |\uone(r)|^{2m} dr\right)\\ +vD^{5}_c \sigma^2_{\two,+} \left(v+\int^t_{t-v} |\uone(r)|^{2m} dr\right)^2\left(\int^t_{t-v} \sigma_{A,-}^2(r) dr\right)^{-2}.
\end{multline*}
\end{prop}
Note that $\sigma_{A,-}^2(t)$ essentially characterizes how well the observation of $\utwo$ is made at time $t$.  It can be zero or close to zero in certain time periods. During such time periods, little information is provided by the upper bound in \cref{prop:obs}.

\begin{proof}[Proof of  \cref{prop:obs}]
 We continue our discussion from the end of the proof of  \cref{prop:control}. For $t-v\leq s\leq t$, the observability Gramian is bounded by
\begin{align*}
\calO_{s,t}&=\int^t_s(\calE_{r,t}^*)^{-1} \bfA^*_1(r,\uone(r)) [\bfSigma_\one\bfSigma^*_\one]^{-1}\bfA_1(r,\uone(r)) \calE_{r,t}^{-1}\\
&\preceq D_c  \int^t_s (|\uone(r)|^{2m}+1) (\calE_{r,t}\calE_{r,t}^*)^{-1} dr\\
&\preceq D^2_c\left(t-s+\int^t_s |\uone(r)|^{2m} dr\right) I_{N_\two}.
\end{align*}
Likewise, we have
\[
\calO_{s,t}\succeq  \int^t_s \sigma_{A,-}^2(r)(\calE_{r,t}\calE_{r,t}^*)^{-1} dr\succeq D^{-1}_c \left(\int^t_s \sigma_{A,-}^2(r) dr\right) I_{N_\two}.
\]
Therefore
\begin{align*}
\calF_{s,t}&=\calO_{s,t}^{-1}\left[\int^t_s (\calE^*_{r,t})^{-1} \calO_{s,r} \bfSigma_\two \bfSigma_\two^* \calO_{s,r} \calE^{-1}_{r,t} dr\right] \calO_{s,t}^{-1}\\
&\preceq \sigma^2_{\two,+} \calO_{s,t}^{-1}\left[\int^t_s (\calE^*_{r,t})^{-1} \calO^2_{s,r} \calE^{-1}_{r,t} dr\right] \calO_{s,t}^{-1}\\
&\preceq D^{2}_c \sigma^2_{\two,+} \left(v+\int^t_s |\uone(r)|^{2m} dr\right)^2\calO_{s,t}^{-1}\left[\int^t_s (\calE_{r,t}\calE^*_{r,t})^{-1}  dr\right] \calO_{s,t}^{-1}\\
&\preceq vD^{5}_c \sigma^2_{\two,+} \left(v+\int^t_s |\uone(r)|^{2m} dr\right)^2\left(\int^t_s \sigma_{A,-}^2(r) dr\right)^{-2} I_{N_\two}.
\end{align*}
In conclusion, by \cref{tmp:Rtbounds} we have $\bfR_\two(t)\preceq g_{v,t}(\uone)I_{N_\two}$, where
\begin{multline*}
g_{v,t}(\uone):=D^2_c\left(v+\int^t_{t-v} |\uone(r)|^{2m} dr\right)\\ +vD^{5}_c \sigma^2_{\two,+} \left(v+\int^t_{t-v} |\uone(r)|^{2m} dr\right)^2\left(\int^t_{t-v} \sigma_{A,-}^2(r) dr\right)^{-2}.
\end{multline*}
\end{proof}

\section{Contraction of the Riccati flow}
\label{sec:contract}
One practical issue in applying the hybrid method is how to initialize $\bfR_\two(0)$, as $p_0(\uone,\utwo)$ often does not  have an explicit form. In fact, the value of $\hat{p}_t$ has a diminishing dependence on $\bfR_\two(0)$ as $t\to \infty$. One way to characterize this is to consider a covariance process $\bfR_\two'(t)$  that follows
\[
d\bfR_\two'=[\bfa_1 \bfR_\two'+\bfR_\two' \bfa_1^*+\bfSigma_\two\bfSigma_\two^*-\bfR_\two' \bfA_1^* (\bfSigma_\one \bfSigma_\one^*)^{-1} \bfA_1 \bfR_\two']dt.
\]
Here $\bfa_1$ and $\bfA_1$ are the same as the one in \eqref{CG_Result}, but $\bfR_\two'(0)\neq \bfR_\two(0)$. In fact, the rescaled difference between $\bfR_\two(t)$ and $\bfR_\two(t)'$ converges geometrically fast, although the contraction rate depends on the realization of $\uone$:
\begin{prop}
\label{prop:converge}
Under the same assumptions of \cref{prop:obs}, for any $t\geq s\geq v$,
\begin{multline*}
\|\bfR_\two(t)-\bfR_\two'(t)\|\leq \sqrt{\|\bfR_\two(t)\|\|\bfR_\two(s)^{-1}\|\|\bfR_\two'(t)\|\|\bfR_\two'(s)^{-1}\|}\\\times\exp\left(-\int^t_s f_{v,r}(\uone)dr\right) \|\bfR_\two(s)-\bfR_\two'(s)\|,
\end{multline*}
where $\|\,\cdot\,\|$ denotes the spectral radius, and
\[
f_{v,t}(\uone):=\sigma_{A,-}^2(t) \left(D^6_c \left(v+\int^t_{t-v} |\uone(r)|^{2m}dr\right) +D_c \right)^{-1}\geq0.
\]
\end{prop}
Note that the same lower and upper bounds in  \cref{prop:control} and \cref{prop:obs} apply to $\bfR_\two(t)$ and $\bfR_\two'(t)$ here, as they are both driven by $\uone$, which is the only thing the bounds depend on.

\begin{rem}
When $\sigma_{A,-}$ is bounded uniformly away from zero and $\bfA_1$ has no dependence on $\uone$, namely $m=1$ in \cref{aspt:control}, then $f_{v,t}$ is bounded away from zero uniformly in time. In this case, one can also show the initialization of $\barut(0)$ has diminishing influence on $\hat{p}_t$, since a $\barut'(t)$ process driven by \eqref{CG_Result} will converge to $\barut(t)$ even if $\barut'(0)\neq \barut(0)$ \cite{bishop2016stability}.  Although in principle the convergence of the mean process should hold in more general scenarios, the convergence rate will have a very involved dependence on the realization of $\uone(s)$. Also it is not a significant property for our estimator, and therefore we omitted here.
\end{rem}

\begin{proof}[Proof of \cref{prop:converge}]
We continue our discussion from the end of the proof of \cref{prop:obs}. In Proposition 3.3 of \cite{bishop2016stability}, it is shown that
\[
\bfR_\two(t)-\bfR_\two'(t)= E_{s,t}(\bfR_\two)[\bfR_\two(s)-\bfR_\two'(s)]E^*_{s,t}(\bfR_\two').
\]
Here $E_{s,t}(\bfR_\two)$ is the solution to
\[
\frac{d}{dt}E_{s,t}(\bfR_\two)=[\bfa_1(t, \uone(t))-\bfR_\two(t) \bfA^*_1(t) [\bfSigma_\one\bfSigma^*_\one]^{-1}\bfA_1(t)]E_{s,t}(\bfR_\two),\quad E_{s,s}(\bfR_\two)=\bfR_\two(s).
\]
Then in the third displayed equation of the proof of theorem 4.8 in \cite{bishop2016stability}, we have the identity
\begin{multline*}
\frac{d}{dt} E^*_{s,t}(\bfR_\two) \bfR_\two(t)^{-1}E_{s,t}(\bfR_\two)=-E^*_{s,t}(\bfR_\two)( \bfR_\two(t)^{-1}\bfSigma_\two\bfSigma_\two^*\bfR_\two(t)^{-1} \\+\bfA^*_1(t) [\bfSigma_\one\bfSigma^*_\one]^{-1}\bfA_1(t))E_{s,t}(\bfR_\two).
\end{multline*}
By  \cref{prop:control}, we have $\bfA^*_1(t) [\bfSigma_\one\bfSigma^*_\one]^{-1}\bfA_1(t)\succeq f_{v,t}(\uone)\bfR_\two(t)^{-1}$, where
\[
f_{v,t}(\uone):=\sigma_{A,-}^2(t) h_{v,t}(\uone).
\]
Consequently,
\[
\frac{d}{dt} E^*_{s,t}(\bfR_\two) \bfR_\two(t)^{-1}E_{s,t}(\bfR_\two)\preceq -f_{v,t}(\uone) E^*_{s,t}(\bfR_\two) \bfR_\two(t)^{-1}E_{s,t}(\bfR_\two).
\]
By Gronwall's inequality, we have
\[
E^*_{s,t}(\bfR_\two) \bfR_\two(t)^{-1}E_{s,t}(\bfR_\two)\preceq  \bfR_\two(s)^{-1}\exp\left(-\int^t_s f_{v,r}(\uone)dr\right),
\]
and therefore,
\[
\|E_{s,t}(\bfR_\two)\|^2=\|E^*_{s,t}(\bfR_\two) E_{s,t} (\bfR_\two)\|\leq \|\bfR_\two(t)\|\|\bfR_\two(s)^{-1}\|\exp\left(-\int^t_s f_{v,r}(\uone)dr\right).
\]
The same inequality also holds for $\|E_{s,t}(\bfR_\two')\|$. In conclusion, we have
\begin{align*}
\|&\bfR_\two(t)-\bfR_\two'(t)\|\leq \|E_{s,t}(\bfR_\two)\| \|E_{s,t}(\bfR_\two')\|\|\bfR_\two(s)-\bfR_\two'(s)\|\\
&\leq  \sqrt{\|\bfR_\two(t)\|\|\bfR_\two(s)^{-1} \|\bfR_\two'(t)\|\|\bfR_\two'(s)^{-1}\|}\exp\left(-\int^t_s f_{v,r}(\uone)dr\right) \|\bfR_\two(s)-\bfR_\two'(s)\|.
\end{align*}
\end{proof}

\section{controllability and long time behavior of conditional Gaussian turbulent dynamical systems with energy-conserving quadratic nonlinearity}\label{Sec:Proof_Stochastic_Model}
\begin{proof}[Proof of  \cref{prop:SF}]
Part 1): Simply note that
\begin{align*}
& \uone\cdot( -\Lambda_{\one,0}\uone+\bfB_{\one,0} (\uone,\uone) +\bfF_\one-\Lambda_{\one,1}\utwo+\bfB_{\one,1}(\uone)\utwo ) \\
&+ \utwo\cdot( -\Lambda_{\two,0}\uone+\bfB_{\two,0} (\uone,\uone)+\bfF_\two)+(-\Lambda_{\two,1}\utwo+\bfB_{\two,1}(\uone)\utwo ))\\
&=- \uone\cdot( \Lambda_{\one,0}\uone+\Lambda_{\one,0}\utwo-\bfF_\one)- \utwo\cdot( \Lambda_{\two,0}\uone+\Lambda_{\two,0}\utwo-\bfF_\two)\\
&\leq -\lambda_- (|\uone|^2+|\utwo|^2)+ \uone\cdot \bfF_\one+  \utwo\cdot \bfF_\two\\
&\leq -\frac12\lambda_- (|\uone|^2+|\utwo|^2)+\frac{1}{2\lambda_-}(|\bfF_\one|^2+|\bfF_\two|^2).
\end{align*}
\newline
\noindent Part 2): With any fixed $\utwo\in \reals^{N_\two}$, $\utwo(t)=\calE_{s,t}\utwo$ follows the ODE
\[
\dot{\utwo}(t)=\bfa_1(t,\uone(t))\utwo(t)=(-\Lambda_{\two,1}+\bfB_{\two,1}(\uone))\utwo(t).
\]
Therefore
\[
\frac{d}{d t}|\utwo(t)|^2=-2 \utwo(t)\cdot \Lambda_{\two,1} \utwo(t)+2\utwo(t)\cdot\bfB_{\two,1}(\uone)\utwo(t)
=-2 \utwo(t)\cdot\Lambda_{\two,1} \utwo(t).
\]
By letting $\uone=0$ and $\utwo=\utwo(t)$ in \eqref{tmp:linear}, the quantity above is in the range of
\[ [-2\lambda_+ |\utwo(t)|^2, -2\lambda_- |\utwo(t)|^2].\]
This leads to the spectrum of $\calE_{s,t}\calE_{s,t}^*$ (and $\calE_{s,t}^*\calE_{s,t}$) to be between $\exp(-2\lambda_+(t-s))$ and $\exp(-2\lambda_-(t-s))$.  Therefore
\[
\calE_{s,t}\bfSigma_\two\bfSigma_\two^*\calE_{s,t}^*\preceq \sigma^2_{\two,+} \calE_{s,t}\calE_{s,t}^*\preceq  \sigma^2_{\two,+} \exp(-2\lambda_-(t-s))I_{N_\two},
\]
\[
\calE_{s,t}\bfSigma_\two\bfSigma_\two^*\calE_{s,t}^*\succeq \sigma^2_{\two,-} \calE_{s,t}\calE_{s,t}^*\succeq  \sigma^2_{\two,-} \exp(-2\lambda_+(t-s))I_{N_\two}.
\]
Consequently,
\[
\calC_{t-1,t}\succeq \sigma^2_{\two,-} \left(\int^t_{t-1}\exp(-2\lambda_+(t-s)) ds \right)I_{N_\two}= \frac{1-\exp(-2\lambda_+)}{2\lambda_+}\sigma^2_{\two,-} I_{N_\two},
\]
and for $t-1\leq s\leq t$,
\[
\calC_{t-s,t}\preceq \calC_{t-1,t}\preceq\sigma^2_{\two,+} \left(\int^t_{t-1}\exp(-2\lambda_-(t-s)) ds \right)I_{N_\two}= \frac{\sigma^2_{\two,+}}{2\lambda_-} I_{N_\two}.
\]
Finally, since we can verify $\|\Lambda_{\one, 1}\|\leq \lambda_+$, by taking $\utwo=0$ in \eqref{tmp:linear}, we find
\[
\|\bfA_1(\uone)\|^2=\|-\Lambda_{\one, 1}+\bfB_{\one, 1}(\uone)\|^2\leq  (\|\Lambda_{\one, 1}\|+\lambda_B \|\uone\|)^2\leq 2\lambda^2_++2\lambda^2_B \|\uone\|^2.
\]
Therefore, $\bfA_1(\uone) (\bfSigma_\one \bfSigma_\one^*)^{-1}\bfA_1^*(\uone)\preceq 2(\lambda^2_++\lambda^2_B \|\uone\|^2)\sigma^{-2}_{\one,-} I_{N_\two}.$
In conclusion,  \cref{aspt:control} holds with
\[
D_c=\max\left\{ 1,  \frac{2\lambda_+\sigma^{-2}_{\two,-}}{1-\exp(-2\lambda_+)}, \frac{\sigma^2_{\two,+}}{2\lambda_-}, 2\lambda^2_+\sigma^{-2}_{\one,-},2\lambda^2_B \sigma^{-2}_{\one,-},\exp(2\lambda_+)\right\}.
\]
\end{proof}

\section{Numerical simulations of the modified triad models}\label{Sec:ModifyModel}
\cref{fig:RegimeI_t20_Noise01} and \cref{fig:RegimeII_t20_Noise01} show the recovered PDFs of the modified triad model \eqref{TriadModel_New} at $t=20$ for Regime I and II, respectively. With the extra small noise $\epsilon$ in the dynamics of $u_1$, the controllability is regained and the skill of the hybrid algorithm is greatly improved.

\cref{fig:RegimeII_damping_t50} shows the statistics and the recovery of the PDFs of the following system
\begin{subequations}\label{TriadModel_New2}
\begin{align}
  \frac{du_1}{dt} &= A_1u_2u_3 - d_1u_1 + \epsilon\dot{W}_1,\label{TriadModel_u1_New2}\\
  \frac{du_2}{dt} &= A_2u_3u_1 - d_2u_2 + \sigma_2\dot{W}_2,\label{TriadModel_u2_New2}\\
  \frac{du_3}{dt} &= A_3u_1u_2 - d_3u_3 + \sigma_3\dot{W}_3,\label{TriadModel_u3_New2}
\end{align}
\end{subequations}
where an extra damping $-d_1u_1$ is imposed in addition to the small noise term in the $u_1$ dynamics. The parameter $d_1=0.1$ and other parameters are the same as those in \cref{TriadModel_New}. The model \eqref{TriadModel_New2} satisfies all the conditions in \cref{prop:SF} and therefore hybrid algorithm is always skillful. Note that the triad model \eqref{TriadModel_New2} is linearly stable with respect to all the three variables and no unbounded growth in the $u_1$ direction as the versions in \eqref{TriadModel_New} and \eqref{TriadModel_New2}.

\begin{figure}[htbp]
  \centering\label{fig:RegimeI_t20_Noise01}
  \hspace*{-2cm}\includegraphics[width=17.5cm]{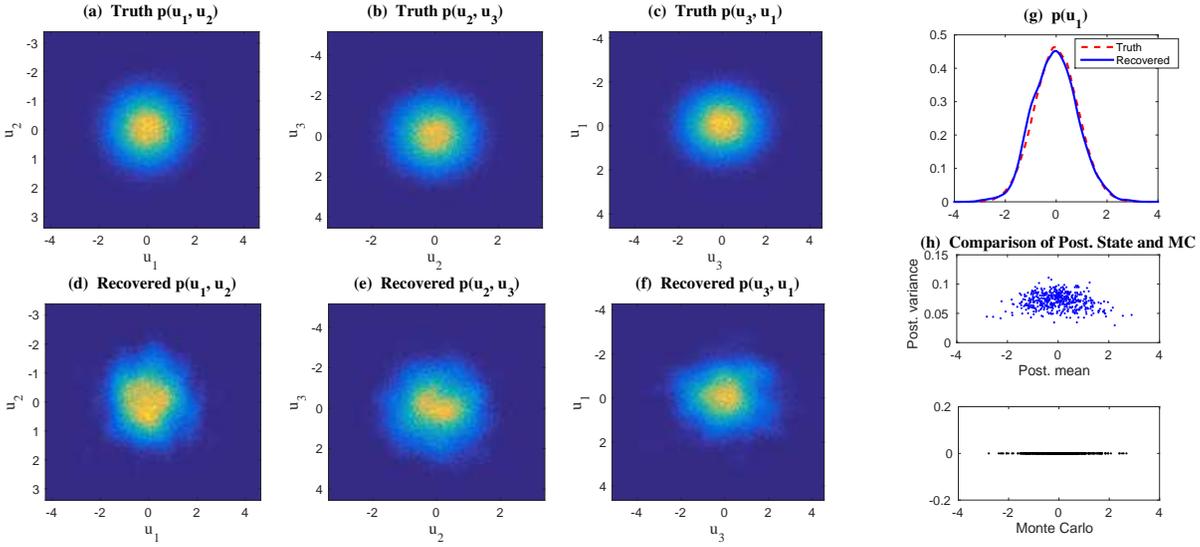}
  \caption{Modified triad model \eqref{TriadModel_New}, Regime I at $t=20$. Same captions as in \cref{fig:RegimeI_t1}.}
\end{figure}

\begin{figure}[htbp]
  \centering\label{fig:RegimeII_t20_Noise01}
  \hspace*{-2cm}\includegraphics[width=17.5cm]{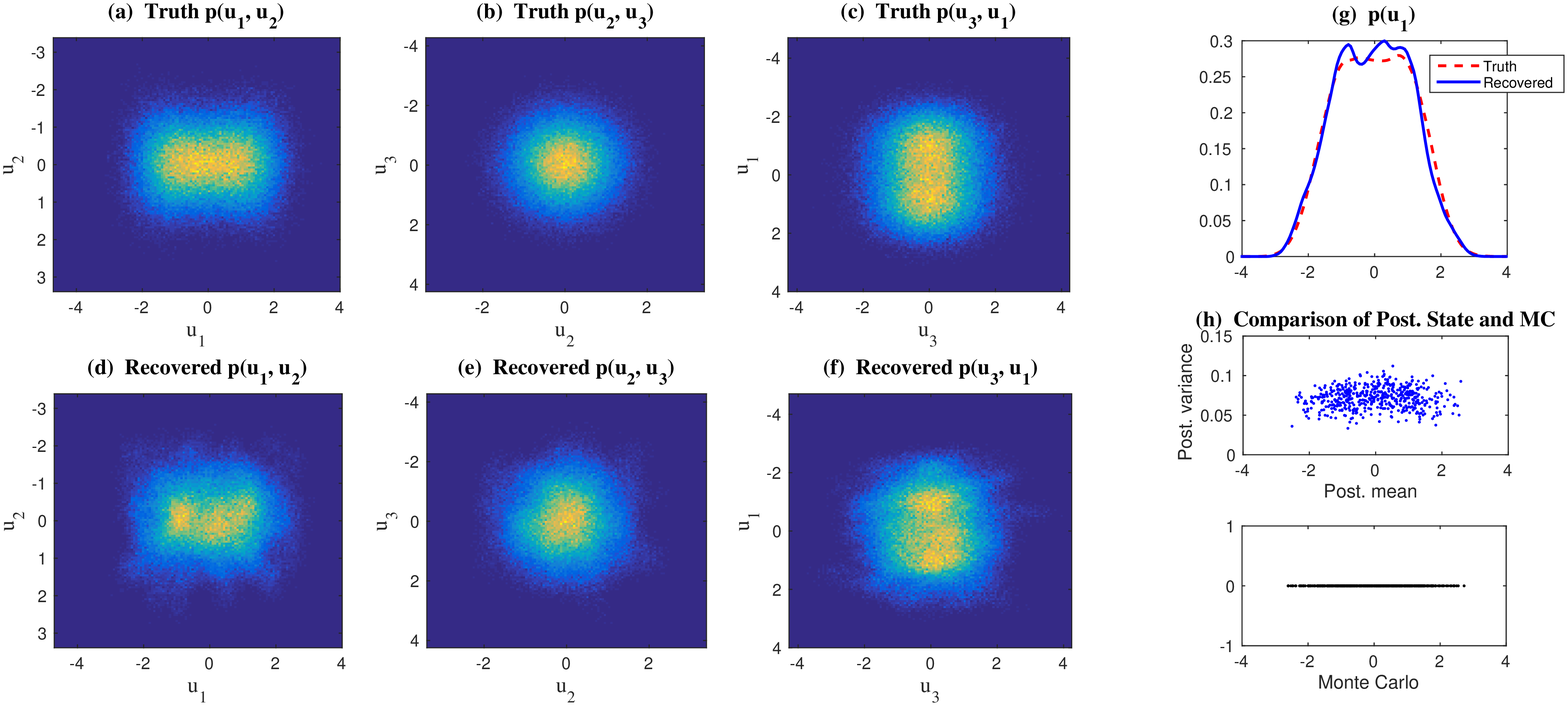}
  \caption{Modified triad model \eqref{TriadModel_New}, Regime II at $t=20$. Same captions as in \cref{fig:RegimeI_t1}.}
\end{figure}

\begin{figure}[htbp]
  \centering\label{fig:RegimeII_damping_t50}
  \hspace*{-2cm}\includegraphics[width=17.5cm]{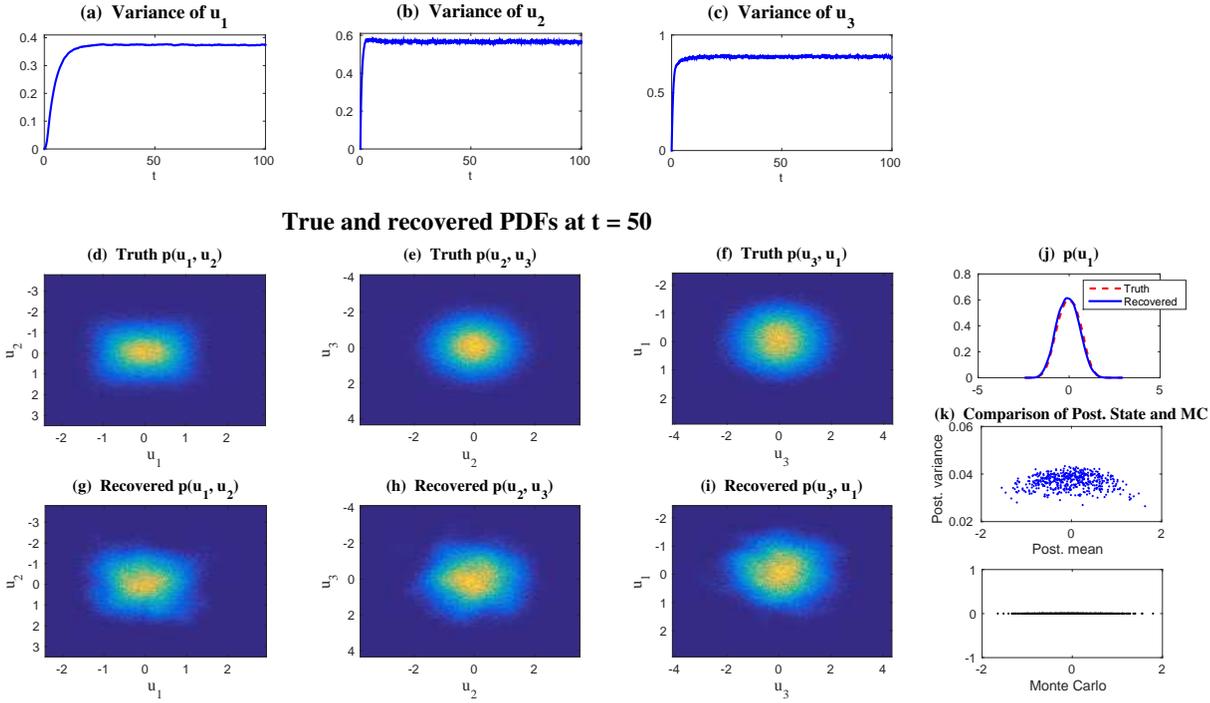}
  \caption{Modified triad model \eqref{TriadModel_New2} with extra damping $d_1u_1$, Regime II at $t=50$. Panels (a)-(c): time evolution of the variance of the three components. Panels (d)-(k): same captions as in \cref{fig:RegimeI_t1}.}
\end{figure}

\bibliographystyle{siamplain}
\bibliography{references}
\end{document}